\documentclass{article}
\usepackage{etex}
\reserveinserts{28}

\usepackage{ytableau}
\usepackage{enumitem}
\usepackage{shuffle,yfonts}
\DeclareFontFamily{U}{shuffle}{}
\DeclareFontShape{U}{shuffle}{m}{n}{ <-8>shuffle7 <8->shuffle10}{}
\newcommand{\sha}{\shuffle}

\usepackage{amsmath,amssymb,amsthm}
\usepackage[all]{xy}
\usepackage{tikz}
\usepackage{mathrsfs}
\usetikzlibrary{decorations.pathreplacing}

\allowdisplaybreaks
\theoremstyle{plain}
\newtheorem{thm}{Theorem}[section]
\newtheorem{lem}[thm]{Lemma}
\newtheorem{cor}[thm]{Corollary}

\newtheorem{prop}[thm]{Proposition}

\theoremstyle{definition}

\newtheorem{eg}[thm]{Example}

 \allowdisplaybreaks

\newcommand{\ot}{{\otimes}}
\newcommand{\ww}{{w}}
\newcommand{\bara}{{\bar a}}
\newcommand{\barb}{{\bar b}}
\newcommand{\barc}{{\bar c}}
\newcommand{\bard}{{\bar d}}
\newcommand{\dep}{{\rm{dp}}}
\newcommand{\barq}{{\bar{q}}}
\newcommand{\baru}{{\bar{u}}}
\newcommand{\barv}{{\bar{v}}}
\newcommand{\barw}{{\bar{w}}}

\newcommand{\breveq}{{\breve{q}}}
\newcommand{\breveu}{{\breve{u}}}
\newcommand{\brevev}{{\breve{v}}}

\newcommand{\tls}{{\tilde{s}}}

\newcommand{\tlt}{{\tilde{t}}}

\newcommand{\tla}{{\tilde{a}}}
\newcommand{\tlb}{{\tilde{b}}}
\newcommand{\tlc}{{\tilde{c}}}
\newcommand{\tld}{{\tilde{d}}}
\newcommand{\db}{{\mathbb D}}
\newcommand{\frakS}{{\mathfrak S}}
\newcommand{\ga}{\alpha}
\newcommand{\gb}{\beta}
\newcommand{\gs}{\sigma}

\newcommand{\ol}{\overline}
\newcommand{\sh}{\shuffle}
\newcommand{\st}{\ast}

\newcommand{\bfs}{{\boldsymbol{\sl{s}}}}
\newcommand{\bft}{{\boldsymbol{\sl{t}}}}

\newcommand{\bfx}{{\boldsymbol{\sl{x}}}}
\newcommand{\bfy}{{\boldsymbol{\sl{y}}}}
\newcommand{\bfz}{{\boldsymbol{\sl{z}}}}
\newcommand{\bfp}{{\bf p}}
\newcommand{\bfq}{{\bf q}}

\newcommand{\bfeta}{{\boldsymbol \eta}}
\newcommand{\bfmu}{{\boldsymbol \mu}}
\newcommand{\bfnu}{{\boldsymbol \nu}}
\newcommand{\eps}{{\varepsilon}}
\newcommand{\bfeps}{{\boldsymbol \eps}}
\newcommand{\bfxi}{{\boldsymbol \xi}}
\newcommand{\gk}{{\kappa}}
\newcommand{\gl}{{\lambda}}

\newcommand{\gF}{{\Phi}}

\def\R{\mathbb{R}}
\def\N{\mathbb{N}}\def\Z{\mathbb{Z}}
\def\Q{\mathbb{Q}}

\def\z{\zeta}

\DeclareMathOperator*{\sgn}{{sgn}}
\DeclareMathOperator\csh{{\text{\makebox{$\,{\scriptstyle \sh}\hskip-9pt\bigcirc$}}}}
\DeclareMathOperator\cshw{{\text{\makebox{$\,\underset{w}{{\scriptstyle \sh}\hskip-8.7pt\bigcirc}$}}}}
\DeclareMathOperator\cst{\circledast}
\DeclareMathOperator\cstw{\underset{\emph{w}}{\circledast}}

\def\R{\mathbb{R}}

\def\N{\mathbb{N}}
\def\Z{\mathbb{Z}}
\def\Q{\mathbb{Q}}

\def\z{\zeta}

\newcommand{\bro}{{\breve1}}
\newcommand{\brt}{{\breve2}}
\newcommand{\II}{{I\!I}}
\newcommand{\III}{{I\!I\!I}}

\begin{document}
\title {\bf Weighted and Restricted Sum Formulas of Euler Sums}
\author{{\sc{Jianqiang Zhao}}\\
\ \ \\
Department of Mathematics, The Bishop's School, La Jolla, CA 92037, USA}
\date{}
\maketitle \noindent{\bf Abstract.} 
One of the most interesting formulas for multiple zeta values is the sum formula proved by Granville and Zagier independently in 1990s. 
Many variations and generalizations of it have been found since then. In this paper, we will provide a uniform approach to proving these results using the regularized double shuffle relations of Euler sums. We summarize these formulas including many new ones with three tables in the Appendix at the end of the paper.




\section{Introduction}
\subsection{Multiple zeta values}
The multiple zeta values have attracted immense interest during the past 30 years due to their
important roles in the studies of the motives, knot theory and many other branches of mathematics as well
as in the computation of Feynman integrals in physics (see, for example, \cite{Brown2012,Broadhurst1996a,BroadhurstKr1997,Zhao2016} or the website on MZVs maintained by M. Hoffman).

For any $\bfs=(s_1,\dots,s_d)\in\N^d$, we define the multiple zeta values (MZVs) by
\begin{equation*}
\z(\bfs):=\sum_{k_1>\cdots>k_d} \frac{1}{k_1^{s_1} \cdots k_d^{s_d}}.
\end{equation*}
It is easy to see that these multiple series converge if and only if $s_1\ge 2$, in which case
we say $\bfs$ is \emph{admissible}. We call $|\bfs|:=s_1+\dotsm+s_d$ the \emph{weight} and $\dep(\bfs):=d$ the \emph{depth}.

In a few correspondences with Goldbach starting from 1742, Euler studied the depth two MZVs,
also called double zeta values. He later published some
of his results in \cite{Euler1776}. Among these, the celebrated decomposition formula states that
for all positive integers  $2\le s\le w-2$,
\begin{equation*}
\sum_{k=2}^{w-1}\left\{{k-1\choose s-1}+{k-1\choose w-s-1}\right\}\zeta(k,w-k)
=\zeta(s)\zeta(w-s).
\end{equation*}
This identity can be regarded as a weighted sum formula for double zeta values. 
For more general MZVs, the well-known sum formula has the form
\begin{equation}\label{equ:MZVsumCFormula}
 \sum_{|\bfs|=w, \dep(\bfs)=d, \bfs \text{ is admissible}} \zeta(\bfs)=\zeta(w) \qquad \forall w>d>0.
\end{equation}
This was proved independently by A.\ Granville \cite{Granville1997b} and D.\ Zagier
using generating functions. We refer the reader to section 5.11 of the author's book \cite{Zhao2016}
for a detailed account of the history of the sum formula and its various generalizations.

The most important properties of MZVs is that they possess two different product structures, which can be derived from
their series representation and their iterated integral representations, respectively. These lead to the
the double shuffle relations, which can be generalized to Euler sums (see next subsection).

In this paper, we will establish many weighted and restricted sum formulas for the Euler sums.
Here, ``restricted'' means that in the formulas some of the components will have prescribed properties/values.
For example, Eie et al. \cite{EieLiOng2009} showed that for any $\ell\in\N$, weight $w=u+\ell$ and
depth $d$
\begin{equation}\label{equ:Eie}
\sum_{|\bfs|=u,\ \dep(\bfs)=d} \z(\bfs,1_\ell)=\sum_{|\bft|=u+\ell,\ \dep(\bft)=\ell+1,\ t_1>u-d} \z(\bft),
\end{equation}
where $t_1$ is the first component of $\bft$ and $1_\ell$ means that $1$ repeats $\ell$ times. The left-hand side is a
restricted sum and the right-hand side has much less terms if $u$ is large and $\ell$ is small.

\subsection{Euler sums}
Euler sums are the alternating version of MZVs. For any $\bfs=(s_1,\dots,s_d)\in\N^d$ and
$\bfmu=(\mu_1,\dots,\mu_d)\in\{\pm1\}^d$, we define
\begin{equation*}
\z(\bfs;\bfmu):=\z\genfrac{(}{)}{0pt}{0}{\bfs}{\bfmu}:=
\sum_{k_1>\cdots>k_d>0} \frac{\mu_1^{k_1}\cdots \mu_d^{k_d} }{k_1^{s_1} \cdots k_d^{s_d}}.
\end{equation*}
This series converges if and only if $(s_1,\mu_1)\ne(1,1)$, in which case
we say $(\bfs;\bfmu)$ is \emph{admissible}. The weight $|\bfs|$ and the depth are defined similarly as with MZVs.
Conventionally, as a shorthand notation we put a bar on top of $s_j$ if $\mu_j=-1$. For example,
\begin{equation*}
\z(2,1;-1,1):=\z(\bar2,1)=\sum_{k>l>0} \frac{(-1)^k }{k^2 l}.
\end{equation*}

For convenience, we now consider a sort of double cover of the set $\N$ of positive integers.
Let $\db$ be the set of \emph{signed numbers}
\begin{equation}\label{equ:dbDefn}
 \db:=\N \cup \ol{\N}, \quad \text{where}\quad \ol{\N}=\{\bar k: k\in\N\}.
\end{equation}
Define the absolute value function $| \cdot |$ on $\db$ by
$|k|=|\bar k|=k$ for all $k\in\N$ and the sign function by
$\sgn(k)=1$ and $\sgn(\bar k)=-1$ for all $k\in\N$.
On $\db$ we define a commutative and associative
binary operation $\oplus$ (called \emph{O-plus})
as follows: for all $a,b\in\db$
\begin{equation}\label{equ:oplusDefn}
    a\oplus b=
\left\{
  \begin{array}{ll}
    \ol{|a|+|b|}, \quad& \hbox{if $\sgn(a)\ne \sgn(b)$;} \\
    |a|+|b|, & \hbox{if $\sgn(a)=\sgn(b)$.} \\
  \end{array}
\right.
\end{equation}
Then the stuffle relation for Euler sums can be expressed more compactly. For example, for all $a,b\in\db\setminus\{1\}$
\begin{equation*}
    \z(a) \z(b)= \z(a,b)+ \z(b,a)+\z(a\oplus b) .
\end{equation*}

Let $M$ be a large positive integer and $\eps>0$ be a very small number. For any composition $\bfs=(s_1,\dots,s_d)\in\N^d$
and $\mu_1,\dotsc,\mu_d=\pm 1$, we consider two variations of the Euler sums. First, set
\begin{equation}\label{equ:MPLMForm}
\z^{(M)}(\bfs;\bfmu):=\z^{(M)}\genfrac{(}{)}{0pt}{0}{\bfs}{\bfmu}:=
\sum_{M\ge n_1>\cdots>n_d>0} \frac{\mu_1^{n_1}\dots \mu_r^{n_r}}{n_1^{s_1} \dots n_d^{s_d}}
\end{equation}
and
\begin{equation}\label{equ:IepsForm}
I^{(\eps)}(\bfs;\bfmu):=I^{(\eps)}\genfrac{(}{)}{0pt}{0}{\bfs}{\bfmu}:=
\int_0^{1-\eps} \left(\frac{d t}{t}\right)^{s_1-1}\frac{d t}{a_1-t} \cdots\left(\frac{d t}{t}\right)^{s_d-1}\frac{d t}{a_d-t}
\end{equation}
where $a_i=\prod_{j=1}^i \mu_j^{-1}$. Then $\z^{(M)}$ satisfies the stuffle product
\begin{equation} \label{equ:MPLstuffleDdep2}
\z^{(M)}(s_1;\mu_1)\z^{(M)}(s_2;\mu_2)=\z^{(M)}(s_1,s_2;\mu_1,\mu_2) + \z^{(M)}(s_2,s_1;\mu_2,\mu_1) +\z^{(M)}(s_1+s_2;\mu_1\mu_2).
\end{equation}
On the other hand, by shuffle product of the iterated integrals we see that
\begin{equation}  \label{equ:2Riem1RiemProddep2}
I^{(\eps)}\genfrac{(}{)}{0pt}{0}{s_1}{\mu_1} I^{(\eps)}\genfrac{(}{)}{0pt}{0}{s_2}{\mu_2} =
\sum_{\substack{t_1,t_2\ge 1\\ t_1+t_2=s_1+s_2}}
\Bigg[\binom{t_1-1}{s_1-1} I^{(\eps)} \genfrac{(}{)}{0pt}{0}{t_1,\ \ t_2\ \ }{\mu_1, \mu_1\mu_2}
+ \binom{t_1-1}{s_2-1} I^{(\eps)} \genfrac{(}{)}{0pt}{0}{t_1,\ \ t_2\ \ }{\mu_2, \mu_1\mu_2}
 \Bigg].
\end{equation}

It is well-known that for admissible $(\bfs;\bfmu)$ we have
\begin{equation*}
\z(\bfs;\bfmu)=\lim_{M\to\infty}\z^{(M)}(\bfs;\bfmu)=\lim_{\eps\to 0^+} I^{(\eps)}(\bfs;\bfmu).
\end{equation*}
Therefore, one can derive the so-called double shuffle relation using the two different product structures, namely, the stuffle and shuffle prodcts. Then by the usual regularization process (cf.~\cite{IKZ2006}) one can discover the extremely useful regularized double shuffle relation. Briefly speaking, for every $\bfs\in\db^d$ (admissible or non-admissible) there are two polynomials of $T$, denoted by $\z_*(\bfs)$ ($*$-regularzed) and $\z_\sh(\bfs;\bfmu)$ ($\sh$-regularzed), such that

\begin{enumerate}[label=(\bf{DBSF\arabic*}),leftmargin=2cm]
  \item $\z_*(\bfs)\z_*(\bfs')$ can be expressed as a $\Q$-linear combination of $*$-regularized Euler sums of weight $|\bfs|+|\bfs'|$ using the stuffle product.
  \item  $\z_\sh(\bfs)\z_\sh(\bfs')$ can be expressed as a $\Q$-linear combination of $\sh$-regularized Euler sums of weight $|\bfs|+|\bfs'|$ using the shuffle product.
  \item \label{page:DBSF3} There is an explicitly defined $\R$-linear map $\rho$
\begin{equation}\label{equ:rhoDefn}
\rho(e^{Tu})=\sum_{j\ge 0} \frac{\rho(T^j)}{j!} u^j=\exp\left( \sum_{n\ge 2}  \frac{(-1)^n\z(n)}{n}  u^n \right) e^{Tu},
\end{equation}
such that $\rho\circ\z_*=\z_\sh$. In particular, for all $\bfs\in\db^d$,
\begin{equation}\label{equ:OneLeading1}
    \z_*(\bfs)=\z_\sh(\bfs) \qquad \text{ if }(s_1,s_2)\ne (1,1).
\end{equation}

\end{enumerate}
We will not go into the details of this theory, instead, we would like to refer the interested reader to \cite[\S13.3.1]{Zhao2016}.

\subsection{Multiple mixed values}
Recently, a few variations of MZVs are suggested and their properties have been investigated. Hoffman
\cite{Hoffman2019} defined an odd variant of MZVs as follows. For admissible $\bfs\in\N^d$, the multiple
$t$-values (MtVs) are defined by
\begin{equation*}
t(\bfs):=\sum_{k_1>\cdots>k_d>0} \frac{(1-(-1)^{k_1}) \cdots (1-(-1)^{k_d})}{k_1^{s_1} \cdots k_d^{s_d}}
=\sum_{k_1>\cdots>k_d>0} \frac{2^d}{(2k_1+1)^{s_1} \cdots (2k_d+1)^{s_d}}.
\end{equation*}
This is slightly different from Hoffman's version because of the 2-powers. It is obvious that MtVs satisfy the so-called stuffle relations. On the other hand, the \emph{multiple $T$-values} (MTVs) are
defined by Kaneko and Tsumura \cite{KanekoTs2019} as follows:
\begin{equation*}
T(\bfs):=\sum_{\substack{k_1>\cdots>k_d>0 \\ k_j\equiv d-j+1 \text{\ (mod 2)} }} \frac{2^d}{k_1^{s_1} \cdots k_d^{s_d}}
=\sum_{k_1>\cdots>k_d>0} \frac{(1+(-1)^{d+k_1})(1-(-1)^{d+k_2}) \cdots (1-(-1)^{k_d})}{k_1^{s_1} k_2^{s_2} \cdots k_d^{s_d}}.
\end{equation*}
They proved that the MTVs can be expressed using iterated integrals and therefore satisfy the shuffle relations
and the duality relations similar to the MTVs.

In a series of papers \cite{XuZhao2020a,XuZhao2020b,XuZhao2020c} Xu and the current author studied the \emph{multiple $S$-values} (MSVs)
defined by
\begin{equation*}
S(\bfs):=\sum_{\substack{k_1>\cdots>k_d>0 \\ k_j\equiv d-j \text{\ (mod 2)} }} \frac{2^d}{k_1^{s_1} \cdots k_d^{s_d}}
=\sum_{k_1>\cdots>k_d>0} \frac{(1-(-1)^{d+k_1})(1+(-1)^{d+k_2}) \cdots (1+(-1)^{k_d})}{k_1^{s_1} k_2^{s_2} \cdots k_d^{s_d}}
\end{equation*}
and more generally, the \emph{multiple mixed values} (MMVs)
\begin{equation*}
M\genfrac{(}{)}{0pt}{0}{\bfs}{\bfeps}:=
\sum_{k_1>\cdots>k_d>0} \frac{(1+\eps_1(-1)^{k_1}) \cdots (1+\eps_d(-1)^{k_d})}{k_1^{s_1} \cdots k_d^{s_d}}.
\end{equation*}
To save space, we put a check on top of $s_j$ (resp. nothing) if $\eps_j=-1$ (resp. $\eps_j=1$)
and say that the signature of $s_j$ is odd (resp. even). For examples,
\begin{equation*}
    T(3,2,1)=M(\breve3,2,\breve1), \qquad t(3,2,1)=M(\breve3,\breve2,\breve1), \qquad S(3,2,1)=M(3,\breve2,1).
\end{equation*}
It is clear that MMVs can be expressed as $\Z$-linear combinations of Euler sums. Observe that
\begin{equation*}
M\genfrac{(}{)}{0pt}{0}{\bfs}{\bfeps}:=\sum_{\bfmu\in\{\pm1\}^d} \Big(\prod_{j=1}^d \sgn(1+\mu_j+\eps_j) \Big)\z(\bfs;\bfmu)
\end{equation*}
where $\sgn(a)=1$ if $a\ge 0$ and $\sgn(a)=-1$ if $a<0$. Hence the alternating sign $\mu_j$ of $\z(\bfs;\bfmu)$ contributes to the sign in front of $\z(\bfs;\bfmu)$  if and only if the signature of $s_j$ is odd. For example,
\begin{align*}
M(3,\brt,\bro)&\, =\z(3,2,1)+\z(\bar3,2,1)-\z(3,\bar2,1)-\z(\bar3,\bar2,1) \\
&\, -\z(3,2,\bar1)-\z(\bar3,2,\bar1)+\z(3,\bar2,\bar1)+\z(\bar3,\bar2,\bar1)
\end{align*}

In general, there are $2^d$ terms when expressing a MMV of depth $d$ in terms of Euler sums.
In order to express the weighted and restricted sum formulas of Euler sums succinctly,
we need to define some incomplete MMVs by selecting some of the terms from the full expansions.
We need some additional notation.
For any index subset $I=\{i_1,\dots,i_r\}\subsetneq [d]:=\{1,2,\dotsc,d\}$
we denote $\bar I= [d]\setminus I$ the complement of $I$ in $[d]$. For $j\in \bar{I}$ we only allow $\eps_j=\pm1$. But for $j\in I$,
$\eps_j$ can be either one of the following three forms:
\begin{equation*}
\eps_j=\pm1, \quad\text{or}\quad \eps_j=\prod_{i\in J_j} e_i,\quad\text{or}\quad \eps_j=\ol{\prod_{i\in J_j} e_i}
\end{equation*}
for some $J_j\subseteq\bar{I}$, where $e_j$'s are formal symbols.
We then define the \emph{incomplete MMVs} by
\begin{equation*}
M_I(\bfs;\bfeps)=M_I\binom{\bfs}{\bfeps}:=\sum_{\mu_j=\pm1\ \forall j\in \bar{I}} \Big(\prod_{j\in \bar{I}} \sgn(1+\mu_j+\eps_j) \Big) \z(\bfs;\bfmu),
\end{equation*}
where for all $j\in I$,
\begin{equation}\label{equ:epsj}
\mu_j=\left\{
  \begin{array}{ll}
  \ \ \  \eps_j, & \hbox{if $\eps_j=1$ or  $\eps_j=-1$ ;} \phantom{\prod_{i\in J_j} }\\
  \   \prod_{i\in J_j} \mu_i, & \hbox{if $\eps_j=\prod_{i\in J_j} e_i$ for some $J_j\subseteq\bar{I}$;} \\
  - \prod_{i\in J_j} \mu_i, \qquad \ & \hbox{if $\eps_j=\ol{\prod_{i\in J_j} e_i}$ for some $J_j\subseteq\bar{I}$.}
  \end{array}
\right.
\end{equation}
It is clear there are $2^{\dep(\bar I)}$ terms of Euler sums in the expansion of $M_I(\bfs;\bfeps)$, for each of which the sign pattern of
every component corresponding to $j\in I$ is prescribed by $\eps_j$. Further, we may use the shorthand notation as follows: 
for $j\in I$, we put a bar on top of $s_j$ if $\eps_j=-1$, or write $s_j\eps_j$ otherwise;
for $j\in \bar{I}$, we put a check on top of $s_j$ if $\eps_j=-1$, or just write $s_j$ if $\eps_j=1$. 
For example, when $d=4$, $I=\{1,3\}$ and $\bar{I}=\{2,4\}$, we have
\begin{equation}\label{equ:M13}
M_{13}(\bar{a},b,c\ol{e_2},\breve{d})=M_{13} \genfrac{(}{)}{0pt}{0}{\  a\  ,\,  b\, ,\, c\, ,\  d\  }{-1,\,1\,,\ol{e_2},-1}
=\z(\bara,b,\barc,d)+\z(\bara,\barb,c,d)-\z(\bara,b,\barc,\bard)-\z(\bara,\barb,c,\bard).
\end{equation}
Here, all the four possible combinations of alternating signs for components $b$ and $d$ can appear.
The sign $-1$ for $a$ means only $\bara$ can appear in the first component while $\ol{e_2}$ means that the
alternating sign for $c$ in each of the Euler sums on the right-hand side of \eqref{equ:M13}
is always opposite to that of the second component $b$. Further, the signs $+1$ for $b$ and $-1$ for $d$ imply that the sign in front of each Euler sum is determined only by whether there is a bar on top of the fourth component $d$ or not, since $\sgn(1+\mu_j+\eps_j)=-1$ only if $\mu_j=\eps_j=-1$.

\section{Generating functions and regularized double shuffle relations of Euler sums}
Define the following two conversion operations (cf. \cite[$\S$13.3.2]{Zhao2016}): for any $\bfmu\in\{\pm1\}^d$,
\begin{equation*}
\bfp(\bfmu):=(\mu_1,\mu_1\mu_2,\dotsc,\mu_1\mu_2\cdots \mu_d),\quad   
\bfq(\bfmu):=(\mu_1,\mu_2\mu_1,\mu_3\mu_2,\dotsc,\mu_d\mu_{d-1}).    
\end{equation*}
Then we know that for admissible $(\bfs;\bfmu)$
\begin{equation*}
\z\genfrac{(}{)}{0pt}{0}{\bfs}{\bfq(\bfmu)}:=\z\big(\bfs; \bfq(\bfmu) \big)=\int_0^1 \left(\frac{dt}{t}\right)^{\tilde{s_1}}\frac{dt}{\mu_1-t}\cdots  \left(\frac{dt}{t}\right)^{\tilde{s_d}}\frac{dt}{\mu_d-t},
\end{equation*}
where for any positive integer $s$ we denote $\tls=s-1$. If $(s_1,\mu_1)=(1,1)$, we can define
the $\sh$-regularized values
$$\z_\sh\genfrac{(}{)}{0pt}{0}{\bfs}{\bfq(\bfmu)}=\z_\sh\big(\bfs; \bfq(\bfmu))\in\R[T].$$

For $d$ variables $\bfx=(x_1,\dots,x_d)$ and $\bfmu=(\mu_1,\dots,\mu_d)\in\{\pm 1\}^d$, define
\begin{align*}
F_\sh(\bfx;\bfmu)=&\sum_{w\ge d} F_\sh^\ww (\bfx;\bfmu), \qquad
F_\sh^\ww (\bfx;\bfmu):=\sum_{\substack{s_1+\dotsm+s_d=w\\ s_1,\dotsc,s_d\in\N}}
    x_1^{\tilde{s_1}}\dotsm x_d^{\tilde{s_d}} \z_\sh(\bfs;\bfmu), \\
S_\sh(\bfx;\bfmu)=&\sum_{w\ge d} S_\sh^\ww (\bfx;\bfmu), \qquad
S_\sh^\ww (\bfx;\bfmu):=\sum_{\substack{s_1+\dotsm+s_d=w\\ s_1,\dotsc,s_d\in\N}}
    x_1^{\tilde{s_1}}\dotsm x_d^{\tilde{s_d}} \z_\sh(\bfs;\bfq(\bfmu) ).
\end{align*}
Thus
\begin{equation*}
F_\sh^\ww(\bfx;\bfmu)=S_\sh^\ww (\bfx;\bfp(\bfmu) ).
\end{equation*}
Consider $d+d'$ variables $\bfx=(x_1,\dots,x_d)$, $\bfy=(y_1,\dots,y_{d'})$ and $\bfmu=(\mu_1,\dots,\mu_{d+d'})\in\{\pm 1\}^{d+d'}$.
Set $i_0:=j_0:=0$. For any partition $I:=\{0\le i_1<\dots<i_r=d\}$ of $\{1,\dots,d\}$
and partition $J:=\{1\le j_1<\dotsm<j_{r-1}\le j_r=d'\}$ of $\{1,\dots,d'\}$, we set
\begin{equation*}
 I_k=\{ i_{k-1}+1, i_{k-1}+2, \dotsc, i_k-1,i_k\}, \quad
  J_k=\{ j_{k-1}+1, j_{k-1}+2,\dotsc, j_k-1,j_k\},   \quad\forall 1\le k\le r.
\end{equation*}
By convention, we set $I_1=\emptyset$ if $i_1=0$ and $J_r=\emptyset$ if $j_{r-1}=d'$. We define
\begin{equation*}
\bfx(I_k)=(x_\ell)_{i_{k-1}<\ell\le i_k}, \qquad \bfy(J_k)=(y_\ell)_{j_{k-1}<\ell\le j_k}
\end{equation*}
It is clear that terms produced by the shuffle product $\bfx\sh \bfy$ are in 1-1 correspondence to
the choices of the pairs of partitions $(I,J)$:
\begin{equation*}
(\bfx;\bfmu) \sh (\bfy;\bfnu):=\bigcup_{(I,J)} \Big\{\big(\bfx(I_1), \bfy(J_1), \dotsc, \bfx(I_r), \bfy(J_r);
\bfmu(I_1),\bfnu(J_1), \dotsc, \bfmu(I_r),\bfnu(J_r)   \big)\Big\}.
\end{equation*}
We now define the \emph{cloning/contaminating operator} $\gk$ by
\begin{multline*}
\gk\big(\bfx(I_1), \bfy(J_1),\bfx(I_2), \bfy(J_2), \dotsc, \bfx(I_r), \bfy(J_r)\big) \\
=\big(\bfx(I_1)+y_1, \bfy(J_1)+x_{i_1+1},\bfx(I_2)+y_{j_1+1}, \bfy(J_2)+x_{i_2+1},  \dotsc, \bfx(I_r)+y_{j_{r-1}+1}, \bfy(J_r)\big)
\end{multline*}
where for any string $(r_1,\dots,r_n)$ we set $(r_1,\dots,r_n)+x:=(r_1+x,\dots,r_n+x)$.
Finally, we set
\begin{equation*}
(\bfx;\bfmu) \csh (\bfy;\bfnu):=\sum_{w\ge 1} (\bfx;\bfmu) \cshw (\bfy;\bfnu), \quad
(\bfx;\bfmu) \cshw (\bfy;\bfnu):=\sum_{(\bfz,\bfxi)\in (\bfx;\bfmu) \sh (\bfy;\bfnu)} S_\sh^\ww (\gk(\bfz);\bfxi).
\end{equation*}
For example, $x \sh y=\{(x,y),(y,x)\}$ and therefore
\begin{align*}
(x;\mu)\cshw (y;\nu)=&\, S_\sh^\ww (x+y,x;\nu,\mu)+S_\sh^\ww (x+y,y;\mu,\nu)\\
=&\, \sum_{\substack{a+b=w\\ a,b\in\N}} (x+y)^{\tla}x^\tlb\z_\sh(a,b;\nu,\mu\nu)+ (x+y)^{\tla}y^\tlb \z_\sh(a,b;\mu,\nu\mu),\\
\end{align*}

\begin{prop}\label{prop:ShuffleRel}
Let $d,d'\in\N$. For any real variables $\bfx=(x_1,\dots,x_d)$, $\bfy=(y_1,\dots,y_d)$, and
alternating signs $\bfmu\in\{\pm1\}^d,\bfnu\in\{\pm1\}^{d'},$ we have
\begin{align*}
F_\sh(\bfx;\bfmu)F_\sh(\bfy;\bfnu)=(\bfx;\bfp(\bfmu)) \csh (\bfy;\bfp(\bfnu)).
\end{align*}
\end{prop}

\begin{proof}
It suffices to prove the following special case when $d'=1$. Let $a=dt/t$ and $b=dt/(1-t)$. Then
\begin{align*}
&\sum_{t\ge1} \sum_{s_1\ge1}\dotsm \sum_{s_d\ge1} \z_\mu(s_1,\dotsc,s_d)\z_\mu(t) \bfx^{\tilde\bfs} y^\tlt \\
=&\,\sum_{t\ge1} \sum_{s_1\ge1}\dotsm \sum_{s_d\ge1} \sum_{r_1+\dotsm+r_d=\tlt} \prod_{j=1}^d \binom{\tilde{s_j}+r_j}{r_j}
    \left(\int_0^{1-\mu} a^{\tilde{s_1}+r_1}b\dotsm a^{\tilde{s_d}+r_d}b   a^{\tlt-r_1-\dotsm-r_d} b \right) \bfx^{\tilde\bfs} y^\tlt \\
=&\,\sum_{k_1\ge1}\dotsm \sum_{k_d\ge1} \sum_{r_1=0}^{\tilde{k_1}}\dotsm\sum_{r_d=0}^{\tilde{k_d}} \prod_{j=1}^d
    \left(\binom{\tilde{k_j}}{r_j} x_j^{\tilde{k_j}-r_j}y^{r_j}\right)
    \sum_{\ell \ge1}\left( \int_0^{1-\mu} a^{\tilde{k_1}}b\dotsm a^{\tilde{k_d}}b  a^{\tilde{\ell}} b \right)  y^{\tilde{\ell}}\\
=&\,\sum_{k_1\ge1}\dotsm \sum_{k_d\ge1}  \prod_{j=1}^d
    \left( \sum_{r_j=0}^{\tilde{k_j}} \binom{\tilde{k_j}}{r_j} x_j^{\tilde{k_j}-r_j}y^{r_j}\right)
    \sum_{\ell \ge1} \left(\int_0^{1-\mu} a^{\tilde{k_1}}b\dotsm a^{\tilde{k_d}}b  a^{\tilde{\ell}} b \right)  y^{\tilde{\ell}} \\
=&\,\sum_{k_1\ge1}\dotsm \sum_{k_d\ge1}  \sum_{\ell \ge1} \z_\mu(k_1,\dotsc,k_d,\ell) (x_1+y)^{\tilde{k_1}}\dotsm (x_d+y)^{\tilde{k_d}}  y^{\tilde{\ell}}.
\end{align*}
The general case follows from exactly the same argument block-wise. By $\sh$-regularization we now can derive the shuffle side of the equation.
The stuffle side follows easily from the definition.
\end{proof}

To study the stuffle side, we first define a convolution of two variables $x$ and $y$. For any positive integer $s\ge 2$, we set
\begin{equation*}
(x\ot y)^{\tls}=(x\ot y)^{s-1}=\sum_{\substack{a+b=s\\ a,b\in\N}} x^\tla y^\tlb.
\end{equation*}
Then terms produced by the stuffle product $\bfx \st \bfy$ are in 1-1 correspondence to the elements in the following set:
\begin{equation*}
(\bfx;\bfmu) \st (\bfy;\bfnu):=\bigcup_{(I,J)} \left(\bigcup_{\circ=``,'' \text{ or } ``\ot''}
\Big\{ \big(\bfx(I_1)\circ\bfy(J_1), \dotsc, \bfx(I_r)\circ\bfy(J_r);
\bfmu(I_1)\circ\bfnu(J_1), \dotsc, \bfmu(I_r)\circ\bfnu(J_r) \big)\Big\}\right),
\end{equation*}
where the (non-commutative) binary operator $\ot$ for two strings of variables is defined by $(u_1,\dotsc,u_a)\ot(v_1,\dotsc,v_b)=(u_1,\dotsc,u_{a-1},u_a\ot v_1,v_2\dotsc,v_b)$,
and for two strings of signs $\pm1$'s the binary operator $\ot$ is defined by $(\z_1,\dotsc,\z_a)\ot(\xi_1,\dotsc,\xi_b)=(\z_1,\dotsc,\z_{a-1},\z_a\xi_1,\xi_2\dotsc,\xi_b)$.
Finally, we set
\begin{equation*}
(\bfx;\bfmu) \cst (\bfy;\bfnu):=\sum_{w\ge 1}(\bfx;\bfmu) \cstw (\bfy;\bfnu), \quad
(\bfx;\bfmu) \cstw (\bfy;\bfnu):=\sum_{(\bfz;\bfxi)\in (\bfx;\bfmu) \st (\bfy;\bfnu)} S_\st^\ww (\bfz;\bfxi)
\end{equation*}
where for $(\bfz;\bfxi)=(z_1,\dotsc,z_d;\xi_1,\dots,\xi_d)$
\begin{equation*}
S_\st^\ww (\bfz;\bfxi)=\sum_{\substack{s_1+\dotsm+s_d=w\\ s_1,\dotsc,s_d\in\N\\ s_j\ge 2\text{ if } z_j=x_r\ot y_t}} z_1^{\tilde{s_1}}\dotsm z_d^{\tilde{s_d}}\z_\st(\bfs;\bfxi).
\end{equation*}
Note that the variable $z$ has the form of $x$ or $y$ or $x\ot y$.
Similarly, we define the generating function of the $*$-regularized MZVs by
\begin{equation*}
F_\ast(\bfx;\bfmu):=\sum_{\bfs=(s_1,\dotsc,s_d) \in\N^d }  x_1^{\tilde{s_1}}\dotsm x_d^{\tilde{s_d}}\z_\st(\bfs;\bfmu)
=\sum_{w\ge d} S_\st^\ww (\bfx;\bfmu)
\end{equation*}
Then the following result follows from the stuffle relations of Euler sums immediately.

\begin{prop}\label{prop:StuffleRel}
Let $d,d'\in\N$. For any real variables $\bfx=(x_1,\dots,x_d)$, $\bfy==(x_1,\dots,x_d)$, and
alternating signs $\bfmu\in\{\pm1\}^d,\bfnu\in\{\pm1\}^{d'},$ we have
\begin{align*}
F_\ast(\bfx;\bfmu)F_\ast(\bfy;\bfnu)=(\bfx;\bfmu) \cst (\bfy;\bfnu).
\end{align*}
\end{prop}

\begin{cor}\label{cor:ESgenFEallw}
Let $d,d'\in\N$. For any $\bfmu\in\{\pm1\}^d,\bfnu\in\{\pm1\}^{d'},$ we have
\begin{align*}
(\bfx;\bfp(\bfmu)) \csh (\bfy;\bfp(\bfnu))-
\sum_{\bfs\in\N^d, \bft\in\N^{d'}} \delta(\bfs,\bfmu;\bft,\bfnu) \z_\sh(\bfs;\bfmu)\z_\sh(\bft,\bfnu)\bfx^{\tilde{\bfs}}\bfy^{\tilde{\bft}}\\
=(\bfx;\bfmu) \cst (\bfy;\bfnu)-
\sum_{\bfs\in\N^d, \bft\in\N^{d'}} \delta(\bfs,\bfmu;\bft,\bfnu) \z_\st(\bfs;\bfmu)\z_\st(\bft,\bfnu)\bfx^{\tilde{\bfs}}\bfy^{\tilde{\bft}},
\end{align*}
where $\delta(\bfs,\bfmu;\bft,\bfnu) =0$ unless $(s_1,s_2)=(\mu_1,\mu_2)=(1,1)$ or  $(t_1,t_2)=(\nu_1,\nu_2)=(1,1)$ when $\delta(\bfs,\bfmu;\bft,\bfnu) =1$.
\end{cor}
\begin{proof}
Observe that both sides of the equation in the corollary
only involve admissible Euler sums $\z(\bfs)$ or regularized values
$\z_\sharp(1,\bfs;1,\bfmu)$ with $(s_1;\mu_1)\ne (1,1)$ ($\sharp=*$ or $\sh$).
The corollary follows from \eqref{equ:OneLeading1}, Propositions \ref{prop:ShuffleRel} and \ref{prop:StuffleRel} quickly.
\end{proof}

Using substitution $x_j\to t x_j$, $y_k\to ty_k$ for all $j$ and $k$ in Corollary~\ref{cor:ESgenFEallw}
and then comparing the coefficients of $t^{w-d-d'}$
we immediately arrive at the following main result of this paper.

\begin{thm}\label{thm:ESgenFE}
Let $w,d,d'\in\N$ such that $w>d+d'$. For any $\bfmu\in\{\pm1\}^d,\bfnu\in\{\pm1\}^{d'},$ we have
\begin{align*}
(\bfx;\bfp(\bfmu)) \cshw (\bfy;\bfp(\bfnu))-
\sum_{|\bfs|+|\bft |=w} \delta(\bfs,\bfmu;\bft,\bfnu) \z_\sh(\bfs;\bfmu)\z_\sh(\bft,\bfnu)\bfx^{\tilde{\bfs}}\bfy^{\tilde{\bft}}\\
=(\bfx;\bfmu) \cstw (\bfy;\bfnu)-
\sum_{|\bfs|+|\bft |=w} \delta(\bfs,\bfmu;\bft,\bfnu) \z_\st(\bfs;\bfmu)\z_\st(\bft,\bfnu)\bfx^{\tilde{\bfs}}\bfy^{\tilde{\bft}},
\end{align*}
where $\delta(\bfs,\bfmu;\bft,\bfnu) =0$ unless $(s_1,s_2)=(\mu_1,\mu_2)=(1,1)$ or  $(t_1,t_2)=(\nu_1,\nu_2)=(1,1)$ when $\delta(\bfs,\bfmu;\bft,\bfnu) =1$.
\end{thm}

We will call the terms involving $\delta(\bfs,\bfmu;\bft,\bfnu)$ in Theorem~\ref{thm:ESgenFE} the modifying terms. At least in the lower depth cases (for example, $d\le 5$) these can all be computed explicitly by Lemma~\ref{lem:relSharp} in a straightforward manner.

\begin{lem}\label{lem:relSharp}
Put $\z_*(\emptyset)=\z_\sh(\emptyset)=1$. Let  $\ell\in\N$, $\bfs\in\db^d$ is admissible or $\bfs=\emptyset$. Then
\begin{equation*}
\z_\sh(1_\ell,\bfs)=\sum_{j=0}^{\ell}\frac{\rho(T^j)|_{T=0} }{j!} \cdot \z_*(1_{\ell-j},\bfs),
\end{equation*}
and
\begin{align}\notag
\sum_{j=0}^{\ell}& \frac{\rho(T^j)|_{T=0} }{j!} u^j = \exp\left( \sum_{n\ge 2}  \frac{(-1)^n \z(n)}{n}u^n  \right)
=1 + \frac{\z(2)}2 u^2 - \frac{\z(3)}3 u^3 + \frac{9\z(4)}{16} u^4 \\
 & -\Big(\frac{\z(5)}5+\frac{\z(2)\z(3)}6\Big)u^5+ \Big(\frac{61\z(6)}{128} + \frac{\z(3)^2}{18} \Big)u^6 -\Big(\frac{\z(7)}7+\frac{\z(2)\z(5)}{10}+\frac{3\z(4)\z(3)}{16}\Big)u^7+\cdots   \label{equ:rhoT=0}
\end{align}
In particular, for all $(a,b)\ne (1,1)$ we have
\begin{equation*}
    \z_\sh(a,b,\bfs)=\z_*(a,b,\bfs).
\end{equation*}
\end{lem}
\begin{proof}
We first prove the case when $\bfs=\emptyset$. Suppose
\begin{equation*}
\z_\sh(1_\ell)=\sum_{j=0}^{\ell} a_j \cdot \z_*(1_{\ell-j}).
\end{equation*}
Applying $\rho$ on both sides we see that
\begin{equation*}
\frac{\rho(T^\ell)}{\ell!}=\sum_{j=0}^{\ell}a_j \cdot \z_\sh(1_{\ell-j})
=\sum_{j=0}^{\ell} a_j\frac{T^{\ell-j}}{(\ell-j)!}.
\end{equation*}
Considering the generating function
\begin{equation*}
\sum_{\ell\ge 0} \frac{\rho(T^\ell)}{\ell!} u^\ell =\sum_{\ell\ge 0}\sum_{j=0}^{\ell}a_j \cdot \z_\sh(1_{\ell-j},\bfs)
=\sum_{\ell\ge 0}\sum_{j=0}^{\ell} a_j\frac{T^{\ell-j}}{(\ell-j)!} u^\ell =\Big(\sum_{j\ge 0} a_ju^j \Big)e^{Tu},
\end{equation*}
we see that the case $\bfs=\emptyset$ of the lemma follows immediately by setting $T=0$. Further,
we deduce the statement \eqref{equ:rhoT=0} from \eqref{equ:rhoDefn}.

In general, we proceed by induction on $\ell$. The case $\ell=1$ is trivial so we may assume $\ell \ge 2$. Then
by induction we see that
\begin{align*}
\z_\sh(1_{\ell+1},\bfs)=&\, \rho(\z_*(1_{\ell+1},\bfs))=
\rho\left(\z_*(1_{\ell+1})\z(\bfs)-\sum_{k=1}^{\ell+1}
\sum_{\substack{\bft\in 1_k*\bfs\\ t_1\ne 1}}\z_*(1_{\ell+1-k},\bft)   \right)\\
=&\,  \z_\sh(1_{\ell+1})\z(\bfs)-\sum_{k=1}^{\ell+1}
\sum_{\substack{\bft\in 1_k*\bfs\\ t_1\ne 1}}\z_\sh(1_{\ell+1-k},\bft)   \\
=&\,   \z_\sh(1_{\ell+1})\z(\bfs)-\sum_{k=1}^{\ell+1}
\sum_{\substack{\bft\in 1_k*\bfs\\ t_1\ne 1}}
\sum_{j=0}^{\ell+1-k} \frac{\rho(T^j)|_{T=0} }{j!} \cdot \z_*(1_{\ell+1-k-j},\bft)  \\
=&\,  \z_\sh(1_{\ell+1})\z(\bfs)-\sum_{j=0}^{\ell}
\frac{\rho(T^j)|_{T=0} }{j!} \cdot \sum_{k=1}^{\ell+1-j} \sum_{\substack{\bft\in 1_k*\bfs\\ t_1\ne 1}}
\z_*(1_{\ell+1-k-j},\bft)  \\
=&\,  \z_\sh(1_{\ell+1})\z(\bfs)-\sum_{j=0}^{\ell+1}
 \frac{\rho(T^j)|_{T=0} }{j!} \Big(\z_*(1_{\ell+1-j})\z(\bfs)-\z_*(1_{\ell+1-j},\bfs) \Big)\\
 =&\,\sum_{j=0}^{\ell+1}\frac{\rho(T^j)|_{T=0} }{j!} \cdot \z_*(1_{\ell+1-j},\bfs).
\end{align*}
by the case $\bfs=\emptyset$. This completes the proof of the lemma.
\end{proof}

\begin{eg} \label{eg:sh2stu}
When $\ell\le 5$ we get
\begin{align*}
\z_\sh(1_2,\bfs)=&\, \z_*(1_2,\bfs)+ \frac{\z(2)}2\z(\bfs), \\
\z_\sh(1_3,\bfs)=&\,\z_*(1_3,\bfs)+ \frac{\z(2)}2\z_*(1,\bfs)-\frac{\z(3)}{3}\z(\bfs), \\
\z_\sh(1_4,\bfs)=&\,\z_*(1_4,\bfs)+\frac{\z(2)}2\z_*(1_2,\bfs)-\frac{\z(3)}{3}\z_*(1,\bfs)  +\frac{9\z(4)}{16}\z(\bfs) , \\ \z_\sh(1_5,\bfs)=&\,\z_*(1_5,\bfs)+\frac{\z(2)}2\z_*(1_3,\bfs)-\frac{\z(3)}{3}\z_*(1_2,\bfs) +\frac{9\z(4)}{16}\z_*(1,\bfs) -\Big(\frac{\z(5)}{5}+\frac{\z(2)\z(3)}{6}\Big)\z(\bfs).
\end{align*}
\end{eg}

For convenience, for any positive integer $j$ and $a_1,\dots,a_{d},b_1,\dots,b_{d'}\in\Z$ we denote by
\begin{equation*}
   \gF^\ww(a_1,\dots,a_d; b_{1},\dots,b_{d'})
\end{equation*}
the weight $w$ part of the identity obtained from the following steps sequentially:
\begin{enumerate}[label=(\arabic*),leftmargin=3cm]
  \item Take $\bfx=(tx_1,\dots,tx_d), \bfy=(ty_{1},\dots,ty_{d'})$;
  \item Compare the coefficient for $t^{w-d-d'}$;
  \item Evaluate at $x_i=a_i$ and $y_j=b_j$ for $1\le i\le d, 1\le j\le d'$.
\end{enumerate}

\begin{eg}
As in Machide's work \cite{Machide2012}, for any subset $I$ of $\{1,\dots,d\}$, we set $x_I=\sum_{i\in i} x_i$. Let $\frakS_d$ be the symmetry group of $d$ letters. In any depth $d$, we have
\begin{align*}
&\, (x_1,\mu_1)\csh  \cdots \csh (x_d,\mu_d)
= \sum_{\gs\in \frakS_d} S_\sh^\ww \Big(x_{\gs(1),\dotsc,\gs(d)},x_{\gs(2),\dotsc,\gs(d)},\dots,x_{\gs(d)};
    \mu_{\gs(1)},\dotsc,\mu_{\gs(d)} \Big)\\
=&\, \sum_{|\bfs|=w, \bfs\in\N^d}  \sum_{\gs\in \frakS_d}
    x_{\gs(1),\dotsc,\gs(d)}^{\tilde{s_1}} x_{\gs(2),\dotsc,\gs(d)}^{\tilde{s_2}} \cdots x_{\gs(d)}^{\tilde{s_d}}
    \z_\sh(\bfs;\mu_{\gs(1)},\mu_{\gs(1)}\mu_{\gs(2)},\dotsc,\mu_{\gs(d)}\mu_{\gs(d-1)}).
\end{align*}
This produces $d!$ terms for each fixed $\bfs$. On the stuffle side, omitting the modifying term (denoted by $\equiv$) we have
\begin{align*}
&\, (x_1,\mu_1)\cst  \cdots \cst (x_d,\mu_d)
\equiv \sum_{\gs\in \frakS_d} \sum_{\overset{2}{\circ},\dotsm,\overset{d}{\circ}} S_\st^\ww \Big(x_{\gs(1)} \overset{2}{\circ} x_{\gs(2)} \overset{3}{\circ} \dotsm \overset{d}{\circ} x_{\gs(d)};
     \mu_{\gs(1)}  \overset{2}{\circ} \mu_{\gs(2)}  \overset{3}{\circ} \dotsm  \overset{d}{\circ} \mu_{\gs(d)}) \Big)
\end{align*}
where $\overset{j}{\circ}=$``,'' unless $\gs(j-1)<\gs(j)$ when $\overset{j}{\circ}=$``,'' or $\ot$.

In the other extreme case, we may produce the fewest $d$ terms in $\gF(x_1,\dotsc,x_{d-1};y)$ by applying
\begin{align*}
&\,  \big(x_1,\dotsc,x_{d-1};\bfp(\mu_1,\dotsc,\mu_{d-1}) \big)\csh  \big(y;\nu \big) \\
=&\,\sum_{j=1}^d  S_\sh^\ww \Big(x_1+y,\dotsc,x_j+y,x_j,\dotsc,x_{d-1};
    \bfp(\mu_1,\dotsc,\mu_{j-1}),\nu,\mu_1\dotsm\mu_j,\dotsc,\mu_1\dotsm\mu_{d-1} \Big)\\
=&\, \sum_{|\bfs|=w, \bfs\in\N^d} \sum_{j=1}^d
\left( \prod_{k=1}^j (x_k+y)^{\tilde{s_k}}  \prod_{k=j+1}^{d} x_{k-1}^{\tilde{s_{k}}}  \right) \z_\sh\Big(\bfs;\mu_1,\dotsc,\mu_{j-1},\nu\mu_1\cdots\mu_{j-1},\nu\mu_1\cdots\mu_j,\mu_{j+1},\dotsc,\mu_{d-1} \Big).
\end{align*}
Here when $j=d$ we understand $\prod_{k=d+1}^{d} x_{k-1}^{\tilde{s_{k}}}$
as $y^{\tilde{s_d}}$, and the component involving $\mu_d$ and thereafter are all vacuous. On the stuffle side,
\begin{align*}
&\,  \big(x_1,\dotsc,x_{d-1};\mu_1,\dotsc,\mu_{d-1}\big)\cst  \big(y;\nu \big) \\
=&\,\sum_{j=1}^d \sum_{\circ=\text{``,'' or $\ot$} }  S_\st^\ww \Big(x_1,\dotsc,x_{j-2},x_{j-1}\circ y,x_j,\dotsc,x_{d-1};
    \mu_1,\dotsc,\mu_{j-2},\mu_{j-1}\circ \nu,\mu_j,\dotsc,\mu_{d-1} \Big).
\end{align*}
Here when $j=1$ there is only one term produced by $x_0\circ y=y$.

We have worked out the details in depth 3 case in \cite{BCJXXZ2020}.
In depth 4, there are four possible combinations. Besides the previous two cases, we can also have the following.
Set $\bfs=(a,b,c,d)$ and $\sum=\sum{}_{\substack{a+b+c+d=w\\ a,b,c,d\in\N}}$. Note there are six terms in $(x,y)\sh (\ga,\gb)$:
\begin{align*}
&\, (x,y;\mu_1,\mu_1\mu_2)\csh (\ga,\gb;\nu_1,\nu_1\nu_2)\\
=&\, S_\sh^\ww (\ga+x,\gb+x,x,y;\nu_1,\nu_1\nu_2,\mu_1,\mu_1\mu_2)
+S_\sh^\ww (\ga+x,x+\gb,\gb+y,y;\nu_1,\mu_1,\nu_1\nu_2,\mu_1\mu_2)\\
&\, + S_\sh^\ww (\ga+x,x+\gb,y+\gb,\gb;\nu_1,\mu_1,\mu_1\mu_2,\nu_1\nu_2)
+S_\sh^\ww (x+\ga,\ga+y,\gb+y,y;\mu_1,\nu_1,\nu_1\nu_2,\mu_1\mu_2)\\
&\, + S_\sh^\ww (x+\ga,\ga+y,y+\gb,\gb;\mu_1,\nu_1,\mu_1\mu_2,\nu_1\nu_2)
+S_\sh^\ww (x+\ga,y+\ga,\ga,\gb;\mu_1,\mu_1\mu_2,\nu_1,\nu_1\nu_2)\\
=&\, \sum \bigg[  (\ga+x)^{\tla} (\gb+x)^\tlb x^\tlc y^\tld \z_\sh(\bfs;\nu_1,\nu_2,\mu_1\nu_1\nu_2,\mu_2)\\
& +(\ga+x)^{\tla} (\gb+x)^\tlb  (\gb+y)^\tlc y^\tld \z_\sh(\bfs;\nu_1,\mu_1\nu_1,\nu_1\nu_2\mu_1,\mu_1\mu_2\nu_1\nu_2) \\
& +(\ga+x)^{\tla} (\gb+x)^\tlb  (\gb+y)^\tlc \gb^\tld \z_\sh(\bfs;\nu_1,\mu_1\nu_1,\mu_2,\mu_1\mu_2\nu_1\nu_2)  \\
&\, + (\ga+x)^{\tla} (\ga+y)^\tlb  (\gb+y)^\tlc y^\tld \z_\sh(\bfs;\mu_1,\nu_1\mu_1,\nu_2,\mu_1\mu_2\nu_1\nu_2)\\
&\, + (\ga+x)^{\tla} (\ga+y)^\tlb  (\gb+y)^\tlc \gb^\tld \z_\sh(\bfs;\mu_1,\nu_1\mu_1,\mu_1\mu_2\nu_1,\mu_1\mu_2\nu_1\nu_2)\\
&\, + (\ga+x)^{\tla} (\ga+y)^\tlb  \ga^\tlc \gb^\tld \z_\sh(\bfs;\mu_1,\mu_2,\mu_1\mu_2\nu_1,\nu_2)\bigg].
\end{align*}
But there are 13 terms on the stuffle side:
\begin{align*}
&\, (x,y;\mu_1,\mu_2)\cst (\ga,\gb;\nu_1,\nu_2)\\
=&\, S_\st^\ww (\ga,\gb,x,y;\nu_1,\nu_2,\mu_1,\mu_2)
+\sum_{\underset{j}{\circ}=\text{``,'' or $\ot$} }  S_\st^\ww (x\underset{1}\circ \ga,y\underset{2}\circ \gb;\mu_1\underset{1}\circ \nu_1,\mu_2\underset{2}\circ \nu_2)\\
&\, +\sum_{\circ=\text{``,'' or $\ot$} } \bigg( S_\st^\ww (\ga,x,y\circ \gb;\nu_1,\mu_1,\mu_2\circ \nu_2)
+S_\st^\ww (x\circ \ga,\gb,y;\mu_1\circ \nu_1,\nu_2,\mu_2)\\
&\, + S_\st^\ww (\ga,x\circ \gb,y;\nu_1,\mu_1\circ \nu_2,\mu_2)
+S_\st^\ww (x,y\circ \ga,\gb;\mu_1,\mu_2\circ \nu_1,\nu_2) \bigg).
\end{align*}
\end{eg}

By using the above and disregarding the modifying term it is easy to find the 12 terms on the shuffle side of the last case:
\begin{align*}
&\, (x;\mu_1) \csh (y;\mu_2)\csh (\ga,\gb;\nu_1,\nu_1\nu_2)\\
=&\, (x+y,y;\mu_1,\mu_1\mu_2)\csh (\ga,\gb;\nu_1,\nu_1\nu_2)
+(x+y,x;\mu_2,\mu_1\mu_2)\csh (\ga,\gb;\nu_1,\nu_1\nu_2)\\
=&\, S_\sh^\ww (\ga,\gb,x+y,y;\nu_1,\nu_1\nu_2,\mu_1,\mu_1\mu_2)
+S_\sh^\ww (\ga,x+y,\gb,y;\nu_1,\mu_1,\nu_1\nu_2,\mu_1\mu_2)\\
&\, + S_\sh^\ww (\ga,x+y,y,\gb;\nu_1,\mu_1,\mu_1\mu_2,\nu_1\nu_2)
+S_\sh^\ww (x+y,\ga,\gb,y;\mu_1,\nu_1,\nu_1\nu_2,\mu_1\mu_2)\\
&\, + S_\sh^\ww (x+y,\ga,y,\gb;\mu_1,\nu_1,\mu_1\mu_2,\nu_1\nu_2)
+S_\sh^\ww (x+y,y,\ga,\gb;\mu_1,\mu_1\mu_2,\nu_1,\nu_1\nu_2)\\
+&\, S_\sh^\ww (\ga,\gb,x+y,x;\nu_1,\nu_1\nu_2,\mu_2,\mu_1\mu_2)
+S_\sh^\ww (\ga,x+y,\gb,x;\nu_1,\mu_2,\nu_1\nu_2,\mu_1\mu_2)\\
&\, + S_\sh^\ww (\ga,x+y,x,\gb;\nu_1,\mu_2,\mu_1\mu_2,\nu_1\nu_2)
+S_\sh^\ww (x+y,\ga,\gb,x;\mu_2,\nu_1,\nu_1\nu_2,\mu_1\mu_2)\\
&\, + S_\sh^\ww (x+y,\ga,x,\gb;\mu_2,\nu_1,\mu_1\mu_2,\nu_1\nu_2)
+S_\sh^\ww (x+y,x,\ga,\gb;\mu_2,\mu_1\mu_2,\nu_1,\nu_1\nu_2)\\
=&\, \sum \bigg[  (\ga+x+y)^{\tla} (\gb+x+y)^\tlb (x+y)^\tlc y^\tld \z_\sh(\bfs;\nu_1,\nu_2,\mu_1\nu_1\nu_2,\mu_2)\\
& +(\ga+x+y)^{\tla} (\gb+x+y)^\tlb  (\gb+y)^\tlc y^\tld \z_\sh(\bfs;\nu_1,\mu_1\nu_1,\nu_1\nu_2\mu_1,\mu_1\mu_2\nu_1\nu_2) \\
& +(\ga+x+y)^{\tla} (\gb+x+y)^\tlb  (\gb+y)^\tlc \gb^\tld \z_\sh(\bfs;\nu_1,\mu_1\nu_1,\mu_2,\mu_1\mu_2\nu_1\nu_2)  \\
&\, + (\ga+x+y)^{\tla} (\ga+y)^\tlb  (\gb+y)^\tlc y^\tld \z_\sh(\bfs;\mu_1,\nu_1\mu_1,\nu_2,\mu_1\mu_2\nu_1\nu_2)\\
&\, + (\ga+x+y)^{\tla} (\ga+y)^\tlb  (\gb+y)^\tlc \gb^\tld \z_\sh(\bfs;\mu_1,\nu_1\mu_1,\mu_1\mu_2\nu_1,\mu_1\mu_2\nu_1\nu_2)\\
&\, + (\ga+x+y)^{\tla} (\ga+y)^\tlb  \ga^\tlc \gb^\tld \z_\sh(\bfs;\mu_1,\mu_2,\mu_1\mu_2\nu_1,\nu_2)\\
&\, + (\ga+x+y)^{\tla} (\gb+x+y)^\tlb (x+y)^\tlc x^\tld \z_\sh(\bfs;\nu_1,\nu_2,\mu_2\nu_1\nu_2,\mu_1)\\
& +(\ga+x+y)^{\tla} (\gb+x+y)^\tlb  (\gb+x)^\tlc x^\tld \z_\sh(\bfs;\nu_1,\mu_2\nu_1,\nu_1\nu_2\mu_2,\mu_1\mu_2\nu_1\nu_2) \\
& +(\ga+x+y)^{\tla} (\gb+x+y)^\tlb  (\gb+x)^\tlc \gb^\tld \z_\sh(\bfs;\nu_1,\mu_2\nu_1,\mu_1,\mu_1\mu_2\nu_1\nu_2)  \\
&\, + (\ga+x+y)^{\tla} (\ga+x)^\tlb  (\gb+x)^\tlc x^\tld \z_\sh(\bfs;\mu_2,\nu_1\mu_2,\nu_2,\mu_1\mu_2\nu_1\nu_2)\\
&\, + (\ga+x+y)^{\tla} (\ga+x)^\tlb  (\gb+x)^\tlc \gb^\tld \z_\sh(\bfs;\mu_2,\nu_1\mu_2,\mu_1\mu_2\nu_1,\mu_1\mu_2\nu_1\nu_2)\\
&\, + (\ga+x+y)^{\tla} (\ga+x)^\tlb  \ga^\tlc \gb^\tld \z_\sh(\bfs;\mu_2,\mu_1,\mu_1\mu_2\nu_1,\nu_2)\bigg].
\end{align*}

Similarly, by using the above and omitting the modifying term it is easy to find the 31 terms on the shuffle side of the last case:
\begin{multline*}
(x;\mu_1) \cst (y;\mu_2)\cst (\ga,\gb;\nu_1,\nu_1\nu_2)\\
\equiv (x,y;\mu_1,\mu_2)\cst (\ga,\gb;\nu_1,\nu_2)
+(y,x;\mu_2,\mu_1)\cst (\ga,\gb;\nu_1,\nu_2)
+(x\ot y;\mu_1\mu_2)\cst (\ga,\gb;\nu_1,\nu_2).
\end{multline*}

As main applications of the various sum formulas of Euler sums, we can often derive the corresponding formulas for MMVs such as
MtVs, MTVs and MSVs. Oftentimes, however, we run into difficulties to disentangle Euler sums of different alternating sign patterns in
Theorem~\ref{thm:ESgenFE}. But the next surprising fact may help us to derive some nice and clean sum formulas for MTVs.
For any $\bfmu=(\mu_1,\dots,\mu_d)\in\{\pm1\}^d$, we define the $T$-sign
\begin{equation*}
    \sgn_T(\bfmu):=\prod_{1\le j\le d, j\equiv d \ (\text{mod}\ 2)} \mu_j= \prod_{j=0}^{[(d-1)/2]} \mu_{d-2j}
\end{equation*}
as the sign in the front of $\z(\bfs;\bfmu)$ when we express $T(\bfs)$ as $\Z$-linear combinations of Euler sums, for any admissible $\bfs\in\N^d$.

\begin{prop} \label{prop:Tsign}
Let $\bfmu\in\{\pm1\}^d$ and $\bfnu\in\{\pm1\}^e$. Then for all $\bfxi\in \bfp(\bfmu) \sh \bfp(\bfnu)$, we have
\begin{equation*}
\sgn_T( \bfq(\bfxi) ) = \sgn_T(\bfmu) \sgn_T(\bfnu).
\end{equation*}
\end{prop}
\begin{proof}
Set $\bfp(\bfmu)=(\mu_1',\dotsc, \mu_d')$ and $\bfp(\bfnu)=(\nu_1',\dotsc, \nu_e')$, namely, $\mu_j'=\prod_{i=1}^j \mu_i$ and
$\nu_k'=\prod_{i=1}^k \nu_i$ for all $j=1,\dotsc,d$ and $k=1,\dotsc,e$.  First we observe that
\begin{equation*}
 \bfq(\bfp(\bfmu),\bfp(\bfnu))=\bfq(\mu_1',\dotsc, \mu_d', \nu_1',\dotsc, \nu_e')=(\mu_1,\dotsc, \mu_d, \mu_d'\nu_1,\nu_2,\dotsc, \nu_e).
\end{equation*}
If $e$ is even, then clearly $ \mu_d'\nu_1'$ does not contribute to the $T$-sign of $\bfq(\bfp(\bfmu),\bfp(\bfnu))$. If $e$ is odd,
then we see that
\begin{equation*}
\sgn_T \bfq(\bfp(\bfmu),\bfp(\bfnu))=\mu_d'\left( \prod_{1\le j\le d, j\not\equiv d \ (\text{mod}\ 2)} \mu_j \right)\sgn_T(\bfnu) = \sgn_T(\bfmu) \sgn_T(\bfnu).
\end{equation*}
So the proposition holds when $\bfxi=\bfxi_0:=(\bfp(\bfmu),\bfp(\bfnu))$. It now suffices to prove that the
$T$-sign does not change if we swap any two adjacent $\mu_j'$ and $\nu_k'$ since every  $\bfxi\in \bfp(\bfmu) \sh \bfp(\bfnu)$
can be obtained by repeating this kind of operation starting from $\bfxi_0$. To be more precise, we have

\medskip
\noindent
\textbf{Claim.} Let $\ell=d+e$. Suppose
\begin{align*}
\bfxi_1=&\, (\gl_1,\dotsc,\gl_i, \mu_j',\nu_k',\gl_{i+3},\dots,\gl_\ell)\in \bfp(\bfmu) \sh \bfp(\bfnu),\\
\bfxi_2=&\, (\gl_1,\dotsc,\gl_i,\nu_k', \mu_j',\gl_{i+3},\dots,\gl_\ell) \in \bfp(\bfmu) \sh \bfp(\bfnu).
\end{align*}
Then we have
\begin{equation}\label{equ:swapTsign}
 \sgn_T \bfq(\bfxi_1)= \sgn_T \bfq(\bfxi_2).
\end{equation}
\medskip

We prove the claim in two different cases.

\begin{enumerate}[label=(\roman*),leftmargin=1cm]
\item  $i=\ell-2$. Then we have
\begin{align*}
\bfq(\bfxi_1)=&\, \bfq(\gl_1,\dotsc,\gl_{\ell-2}, \mu_d',\nu_e')=(\gl_1,\gl_1\gl_2,\dotsc,\gl_{\ell-3}\gl_{\ell-2}, \gl_{\ell-2}\mu_d',\mu_d'\nu_e'),\\
\bfq(\bfxi_2)=&\, \bfq(\gl_1,\dotsc,\gl_{\ell-2},\nu_e', \mu_d')=(\gl_1,\gl_1\gl_2,\dotsc,\gl_{\ell-3}\gl_{\ell-2}, \gl_{\ell-2}\nu_e',\mu_d'\nu_e').
\end{align*}
So clearly the claim holds in this case.

\item  $i<\ell-2$. Then
\begin{align*}
\bfq(\bfxi_1)=&\, (\gl_1,\gl_1\gl_2,\dotsc,\gl_{i-1}\gl_{i}, \gl_{i}\mu_j',\mu_j'\nu_k',\nu_k'\gl_{i+3},
        \gl_{i+3}\gl_{i+4},\dotsc,\gl_{\ell-1}\gl_\ell),\\
\bfq(\bfxi_2)=&\, (\gl_1,\gl_1\gl_2,\dotsc,\gl_{i-1}\gl_{i}, \gl_{i}\nu_k',\nu_k'\mu_j',\mu_j'\gl_{i+3},
        \gl_{i+3}\gl_{i+4},\dotsc,\gl_{\ell-1}\gl_\ell).
\end{align*}
If $i\equiv \ell \pmod{2}$, then neither $\gl_{i}\mu_j'$ nor $\nu_k'\gl_{i+3}$ appears in $T$-sign of $\bfq(\bfxi_1)$.
Similarly, neither $\gl_{i}\nu_k'$ nor $\mu_j'\gl_{i+3}$ appears in $T$-sign of $\bfq(\bfxi_2)$. Hence the claim holds.
If $i\not\equiv \ell\pmod{2}$, then the claim also holds since $\gl_{i}\mu_j'\cdot \nu_k'\gl_{i+3}=\gl_{i}\nu_k'\cdot \mu_j'\gl_{i+3}$.
\end{enumerate}

This completes the proof of the claim and therefore the proposition.
\end{proof}

\section{Restricted sum formulas}\label{sec:restSumForm}
Fix depth $d\in\N$, weight $w\ge d+1$, and alternating signs $\bfmu=(\mu_1,\dotsc,\mu_d)$.
In Theorem~\ref{thm:ESgenFE}, the left-hand (resp. right-hand) side produces a $\Z$-linear combination of $S_\sh^\ww $ (resp. $S_\st^\ww $). Note that for all $k,l\ge 0$,
\begin{align*}
S_\sh^\ww \big(0_{k},1,0_l; \bfmu\big)=&\, \z_\sh\big(1_{k},w-k-l,1_l; \bfq(\bfmu)\big),\\
S_\st^\ww \big(0_{k},1,0_l; \bfmu\big)=&\, \z_\st\big(1_{k},w-k-l,1_l; \bfmu\big).
\end{align*}
Consider $\gF^\ww(0_k,1,0_l;0)$ in Theorem~\ref{thm:ESgenFE}.
Here, we illustrate the idea using only the MZVs where all the alternating signs $\mu_j=1$.
Omitting the modifying terms (denoted by $\equiv$) we get
\begin{multline}\label{equ:ab1...1}
S_\sh^\ww \big(0_{k},1,1,0_{l}\big)+k S_\sh^\ww \big(0_k,1,0_{l}\big)+(l)S_\sh^\ww \big(0_{k},1,0_{l}\big) \\
\equiv
S_\st^\ww \big(0_{k},(0)\ot(1),0_{l}\big)+\sum_{j=0}^{k-1} S_\st^\ww \big(0_j,(0)\ot(0),0_{k-1-j},1,0_{l}\big)
+\sum_{j=0}^{l-1} S_\st^\ww \big(0_{k},1,0_j,(0)\ot(0),0_{l-1-j}\big).
\end{multline}
Note that $(0)\ot(0)$ (resp.\ $(0)\ot(1)$) means the corresponding argument can only be 2 (resp.\ at least 2). The general Euler sum case can be dealt with similarly since there is only one term of the form $S_\sh^\ww\big(0_{k},1,1,0_{l}\big)$ with two components equal to 1 in \eqref{equ:ab1...1}. Hence we can produce the sum formula for
\begin{equation}\label{equ:zMiddle2termSum}
\sum_{a+b=w-k-l} \z_\sh\big(1_{k},a,b,1_{l}; \bfmu\big)
\end{equation}
for all choices of signs $\bfmu$. Using Lemma~\ref{lem:relSharp} we can easily find the sum formula  for
\begin{equation}\label{equ:z2termSum}
\sum_{a+b=w-d} \z_\sh\big(a,b,1_d; \bfmu\big) \text{ and }\sum_{a+b=w-d} \z_\st\big(a,b,1_d; \bfmu\big).
\end{equation}
This is the finite double shuffle relation when multiplying two Euler sums $\z_\sharp\big(w-d,1_{d-1}; \bfxi\big)$ with $\z_\sharp(1;\nu)$ in two different ways. Each of the $2d+1$ terms on the stuffle side is explicit. By different combinations of \eqref{equ:z2termSum} with all possible choices of $\bfmu\in\{\pm1\}^d$ one obtains the sum formulas for
\begin{equation}\label{equ:M2termSum}
\sum_{a+b=w-d,a\ge2} F\big(a,b,1_d; \bfmu\big),
\end{equation}
where $F$ can be either MTVs, or MtVs, or MSVs. Here we can remove all the non-admissible terms by subtracting $\z_\sh\big(1,w-d-1,1_d;\bfmu\big)$ from \eqref{equ:z2termSum} when $k=0$.

As applications of the above procedure, we find the following restricted weighted sum formulas of Euler sums.
\begin{thm} \label{thm:zab1_d}
For all $u\ge 3$, let $q=u-1$ and $p=u-2$. Then for all $d\ge 0$ we have
\begin{equation*}
\sum_{a+b=u,\, a,b\in\N}   \z_\sh(a,b,1_d;\bfmu)=I_{\ref{thm:zab1_d}}+\II_{\ref{thm:zab1_d}}-\III_{\ref{thm:zab1_d}}.
\end{equation*}
where by setting $\nu_2=\mu_1\mu_2$ and $\nu_j=\mu_j$ for all $j\ne 2$
\begin{align*}
I_{\ref{thm:zab1_d}}=&\, \z_*(1,p,1_d;\nu_1,\nu_2,\dots,\nu_{d+2})+\sum_{j=0}^d \z(q,1_{d+1};\nu_2,\dots,\nu_{j+2},\nu_1,\nu_{j+3},\dots,\nu_{d+2}),\\
\II_{\ref{thm:zab1_d}}=&\, \z(u,1_d;\nu_1\nu_2,\nu_3,\dots,\nu_{d+2})+\sum_{j=0}^{d-1} \z(q,1_j,2,1_{d-1-j};\nu_2,\dots,\nu_{j+2},\nu_1\nu_{j+3},\nu_{j+4},\dots,\nu_{d+2}),\\
\III_{\ref{thm:zab1_d}}=&\,\sum_{j=0}^d   \z(q,1_{d+1};\nu_2,\dots,\nu_{j+2},\nu_1\dots\nu_{j+2},\nu_1\dots\nu_{j+3}, \nu_{j+4},\dots,\nu_{d+2}).
\end{align*}
In particular, by setting $\bfmu_2=(\mu_2,\dots,\mu_{d+2})$ if $\mu_1=1$ we have
\begin{multline*}
\sum{}'\z(a,b,1_d;1,\bfmu)=  \z(u,1_d;\bfmu_2)
+\sum_{j=0}^d \z(q,1_{d+1};\mu_2,\dots,\mu_{j+2},\mu_1,\mu_{j+3},\dots,\mu_{d+2})\\
 +\sum_{j=0}^{d-1} \z(q,1_j,2,1_{d-1-j};\bfmu_2)
-\sum_{j=0}^d   \z(q,1_{d+1};\mu_2,\dots,\mu_{j+2},\mu_2\dots\mu_{j+2},\mu_2\dots\mu_{j+3}, \mu_{j+4},\dots,\mu_{d+2}).
\end{multline*}
\end{thm}
\begin{proof} Take $\bfx=(1,0_d;\nu_2,\dots,\nu_{d+2})$ and $\bfy=(0;\nu_1)$.
Observe that the modifying terms have no contribution
to $\gF^w(1,0_d;0)$ since the weight $w\ge d+3$. Next, we see that the shuffle side is equal to
\begin{align*}
\sum_{a+b=u,\, a,b\in\N}  \z_\sh(a,b,1_d;\bfmu)+\III_{\ref{thm:zab1_d}}
\end{align*}
and the stuffle side is given by $I_{\ref{thm:zab1_d}}+\II_{\ref{thm:zab1_d}}$ where the terms in $\II_{\ref{thm:zab1_d}}$ are produced by ``stuffing''. The theorem follows immediately.
\end{proof}

For future reference we record the following corollary of Theorem~\ref{thm:zab1_d} in depth 4.
Note that the first formula is the special case of \eqref{equ:Eie} when $\ell=d=2$.
\begin{cor} \label{cor:zab11}
For all $u\ge 3$, let $q=u-1$.  Set $\sum{}'=\sum_{a+b=u,a,b\in\N, a\ge2}$. Then
\begin{align*}
\sum{}'\z(a,b,1,1)=&\, \z(u,1_2)+\z(q,2,1)+\z(q,1,2),\\
\sum{}'\z(a,\bar{b},1,1)=&\, 4\z(\barq,1_3)-M_{14}(\barq,1,1,e_2e_3)+\z(\baru,1_2)+\z(\barq,2,1)+\z(\barq,1,2),\\
\sum{}'\z(a,b,\bar1,1)=&\, 2\z(q,\bar1,1_2)-M_{124}(q,\bar1,1,\bar1)+\z(u,\bar1,1)+\z(q,\bar2,1)
+\z(q,\bar1,2),\\
\sum{}'\z(a,b,1,\bar1)=&\, M_{123}(q,1,\bar1,\bro)+\z(u,1,\bar1)+\z(q,2,\bar1)+\z(q,1,\bar2),\\
\sum{}'\z(a,b,\bar1,\bar1)=&\, M_{124}(q,\bar1,\bro,\ol{e_3})+\z(u,\bar1_2)+\z(q,\bar2,\bar1)+\z(q,\bar1,\bar2),\\
\sum{}'\z(a,\bar{b},1,\bar1)=&\, 2\z(\barq,1,1,\bar1)-M_{134}(\barq,1,\bar1,e_2)+\z(\baru,1,\bar1)
+\z(\barq,2,\bar1)+\z(\barq,1,\bar2),\\
\sum{}'\z(a,\bar{b},\bar1,1)=&\, \z(\barq,1,\bar1,1)-\z(\barq,\bar1,1_2)+\z(\baru,\bar1,1)+\z(\barq,\bar2,1)+\z(\barq,\bar1,2),\\
\sum{}'\z(a,\bar{b},\bar1,\bar1)=&\, M_{134}(\barq,1,\bar1,\ol{e_2})-M_{124}(\barq,\bar1,1,\bar1)+\z(\baru,\bar1_2)
+\z(\barq,\bar2,\bar1)+\z(\barq,\bar1,\bar2).
\end{align*}
\end{cor}

Note that we have removed the non-admissible terms from the sums. But if the leading component has an alternating sign then all the terms are admissible. Thus we obtain the next theorem.
\begin{cor} \label{cor:zbarab11}
For all $u\ge 3$, let $q=u-1$.  Set  $\sum=\sum_{a+b=u,a,b\in\N}$. Then
\begin{align*}
\sum{}\z(\bara,b,1,1)=&\,\z(\bar1)\z(\barq,1,1)-3\z(\barq,1,1,1),\\
\sum{}\z(\bara,b,1,\bar1)=&\,\z(\bar1)\z(\barq,1,\bar1)-2\z(\barq,1,1,\bar1)-\z(\barq,1,\bar1,\bar1),\\
\sum{}\z(\bara,b,\bar1,1)=&\,\z(\bar1)\z(\barq,\bar1,1)-\z(\barq,1,\bar1,1)-\z(\barq,\bar1,\bar1,\bar1)-\z(\barq,\bar1,1,\bar1),\\
\sum{}\z(\bara,\barb,1,1)=&\,\z(\bar1)\z(q,1,1)-\z(q,\bar1,\bar1,1)-\z(q,1,\bar1,\bar1)-\z(q,1,1,\bar1),\\
\sum{}\z(\bara,\barb,\bar1,1)=&\,\z(\bar1)\z(q,\bar1,1)-3\z(q,\bar1,1,1),\\
\sum{}\z(\bara,\barb,\bar1,\bar1)=&\,\z(\bar1)\z(q,\bar1,\bar1)-2\z(q,\bar1,1,\bar1)-\z(q,\bar1,\bar1,\bar1),\\
\sum{}\z(\bara,b,\bar1,\bar1)=&\,\z(\bar1)\z(\barq,\bar1,\bar1)-\z(\barq,1,\bar1,\bar1)-2\z(\barq,\bar1,\bar1,1),\\
\sum{}\z(\bara,\barb,1,\bar1)=&\,\z(\bar1)\z(q,1,\bar1)-\z(q,\bar1,\bar1,\bar1)-2\z(q,1,\bar1,1).
\end{align*}
\end{cor}

Corollaries \ref{cor:zab11} and \ref{cor:zbarab11} immediately imply the following results concerning MTVs, MtVs and MSVs.
We leave the detailed computation to the interested readers.
\begin{thm} \label{thm:Tab11}
For all $u\ge 3$, let $q=u-1$. Set $\sum{}'=\sum_{a+b=u,a,b\in\N, a\ge2}$. Then
\begin{align*}
\sum{}' T(a,b,1,1)= &\, 2\Big(T(q,2,1)-\z(\bar1)T(q,1_2)-M_1(q\ol{e_2e_3e_4},\breve{1},1,\breve{1})\Big)
    -4M_1(q\ol{e_2},\breve{1},1,\breve{1}),\\
\sum{}' t(a,b,1,1)=&\,2\Big(-t(q,1_3)+M_1(\bar1,\breve{q},\breve{1},\breve{1})+M_3(\breve{q},\breve{1},e_2,\breve{1})\Big), \\
\sum{}' S(a,b,1,1)= &\, 2\Big(S(q,2,1)-M_2(\breve{q},\ol{e_1e_2},\breve{1},1)\Big)+4\Big(\z(\bar{q},\bar2,\bar1)-M_{34}(q,\breve{1},\bar1_2)
        -M_{24}(\breve{q},\ol{e_1},\breve{1},1)\Big).
\end{align*}
\end{thm}

More generally, the identity $\gF\big(0_{k},1,0_l;0,0\big)$ in Theorem~\ref{thm:ESgenFE} yields the sum formula  for
\begin{equation}\label{equ:23termSum}
\sum_\eta\sum_{a+b=w-k-l-1} \z_\sh\big(1_{k},a,b,1_{l+1}; \bfeta\big)
+\sum_\nu\sum_{a+b=w-k-l-1} \z_\sh\big(1_{k+1},a,b,1_{l}; \bfnu\big)
+\sum_{a+b+c=w-d+3} \z_\sh\big(1_{k},a,b,c,1_{l}; \bfmu\big),
\end{equation}
which are the results of inserting the one or two $0$'s of $\bfy$ in front of the
$1$ in $\bfx$ during the shuffling process on the left-hand side of Theorem~\ref{thm:ESgenFE}.
With \eqref{equ:zMiddle2termSum} we immediately derive the sum formula for
\begin{equation}\label{equ:3termSum}
\sum_{a+b=w-k-l} \z_\sh\big(1_{k},a,b,c,1_{l}; \bfmu\big).
\end{equation}
For any $k,l\ge 0$, repeating this process by setting $\bfx=\big(0_{k},1,0_l\big)$ and
$\bfy=\big(0_d\big)$ and applying Lemma~\ref{lem:relSharp},
we can find the restricted sum formulas for
\begin{equation} \label{equ:ElltermSum}
\sum_{\bfs\in\N^d, |\bfs|=w} \z_\sh\big(1_{k},\bfs,1_l; \bfmu\big), \qquad
\sum_{\bfs\in\N^d, |\bfs|=w} \z_\st\big(1_{k},\bfs,1_l; \bfmu\big).
\end{equation}

We can also derive restricted sum formula with exactly two arguments restricted to 1's.
For simplicity, we illustrate this using MZVs again. First,
using $\gF\big(0_{k},1,1,0_l;0\big)$ in Theorem~\ref{thm:ESgenFE} and
omitting the modifying term (denoted by $\equiv$ again) we see that
\begin{align*}
&\, kS_\sh^\ww \big(0_{k+1},1,1,0_{l}\big)+2S_\sh^\ww \big(0_{k},1,1,1,0_{l}\big)
 +(l+1) S_\sh^\ww \big(0_{k},1,1,0_{l+1}\big) \\
\equiv &\, (k+1) S_\st^\ww \big(0_{k+1},1,1,0_{l}\big)+S_\st^\ww \big(0_{k},1,0,1,0_{d-k-3}\big)
+(l+1) S_\st^\ww \big(0_{k},1,1,0_{l+1}\big)+S_\st^\ww \big(0_{k},(0)\ot(1),1,0_{l}\big)\\
+&\,S_\st^\ww \big(0_{k},1,(0)\ot(1),0_{l}\big) + \sum_{j=1}^k  S_\st^\ww \big(0_{j-1},(0)\ot (0),0_{k-j},1,1,0_{l}\big)
+ \sum_{j=1}^{l} S_\st^\ww \big(0_{k},1,1,0_{j-1},(0)\ot (0),0_{l-j}\big) .
\end{align*}
Together with \eqref{equ:23termSum} this leads to the sum formula for the first sum below:
\begin{equation}\label{equ:z1_ka1b1_l}
\sum_{a+b=w-k-l-1} \z_\st\big(1_{k},a,1,b,1_{l}; \bfmu\big), \qquad
\sum_{a+b=w-k-l-1} \z_\sh\big(1_{k},a,1,b,1_{l}; \bfmu\big).
\end{equation}
We can derive the sum formula for the second sum above if we apply Lemma \ref{lem:relSharp}.
\begin{thm} \label{thm:Ta1c1}
For all $u\ge 3$, let $q=u-1$.  Set $\sum{}'=\sum_{a+b=u,a,b\in\N, a\ge2}$. Then
\begin{align*}
&\, \sum{}' T(a,1,b,1)=4\Big(\z(\bar1)(M(q,\breve{1},\breve{1})+T(q,2))+T(q,3)+M(q,\breve{1},\breve{2})+M(q,\breve{2},\breve{1})\Big)\\
    &\, \hskip3cm    -2S(q,2,1)+2M(q,\breve{1},\breve{1},\breve{1})-2M(q,\breve{1},1_2)+8M_{34}(q,\breve{1},\bar1_2),\\
&\, \sum{}' t(a,1,b,1)=2\Big(M(\breve{u},\breve{1},1)-M_3(\breve{q},\breve{1},\bar1,\breve{1})-\z(\bar2)t(q,1)\Big)-4\Big(\z(\bar1)t(u,1)+t(v,1)+t(u,2)\Big),\\
&\, \sum{}' S(a,1,b,1)=2\Big(\z(\bar1)^2 S(q,1)+M(\breve{q},1,2)-T(q,2,1)-\z(2)\z(q,1)+M_{34}(\breve{q},1,\bar1,1)\Big)\\
&\,\hskip1cm +4(M(\breve{q},2,1)+M_{34}(\breve{q},1,1,\bar1)-\z(\bar1)S(q,2)+2S(q,3)-\z(\bar1)M(\breve{q},1_2) \Big)+6M_{34}(\breve{q},1,\bar1_2).
\end{align*}
\end{thm}

More generally, for any $\ell\in\N$, using  $\gF\big(0_{k},1,1,0_l;0_d\big)$ in Theorem~\ref{thm:ESgenFE}
we can recursively derive the sum formulas for
\begin{equation}\label{equ:z1_ka1_db1_l}
\sum_{a+b=w-d-k-\ell} \z_\st\big(1_{k},a,1_d,b,1_{\ell}; \bfmu\big), \qquad
\sum_{a+b=w-d-k-\ell} \z_\sh\big(1_{k},a,1_d,b,1_{\ell}; \bfmu\big).
\end{equation}
Here, to derive the second sum formula, we need to use \eqref{equ:ElltermSum}, Lemma \ref{lem:relSharp} and a recursive procedure.

\begin{thm} \label{thm:Ma11d}
For all $u\ge 3$, let $q=u-1$ and $v=u+1$.  Set $\sum{}'=\sum_{a+b=u,a,b\in\N, a\ge2}$. Then
\begin{align*}
&\, \sum{}' T(a,1,1,b)=2\Big( M(q,\breve{1},\breve{2})-\z(2)S(q,1)+(2\z(2)\z(\bar1)-3\z(\bar3))T(q)\Big) \\
&\,\hskip3cm  +4\Big(M_4(\breve{q},1,\breve{1},\ol{e_3})-\z(\bar1)T(q,1_2)-M_2(\breve{q},\bar2e_1,\breve{1})\Big)-6M_3(q,\breve{1},\bar1,\breve{1})
, \\
&\, \sum{}' t(a,1,1,b)=2\Big(T(u,1_2)-\z(2)T(u)+ M_{12}(\breve{q},\bar1,\breve{1},\breve{1})-M_2(\breve{q},\bar1,\breve{1},\breve{1})\Big)\\
&\,\hskip3cm +\z(3)T(q)+8\z(\bar1)T(v)+8T(w)-4S(v,1) , \\
&\, \sum{}' S(a,1,1,b)=2\Big(\z(2)T(q,1)+M(\breve{q},1_3)-T(q,1,2)\Big)+8\Big(M_{13}(e_2q,\breve{1},\ol{e_4},\breve{1})-M_{34}(\breve{q},1,\bar1_2)\Big)\\
&\,\hskip1cm +4\Big(\z(\bar1)M(q,\bar1_2)+\z(\bar1)T(q,2)-M_3(q,\breve{2},e_1e_2)\Big)
+6T(q,2,1)-M(\breve{q},\breve{1},1,\breve{1})-M(\breve{q},1,\breve{1},\breve{1}).
\end{align*}
\end{thm}

If we use $\gF(0,0,1;0)$ we may derive the restricted sum formula of the following type for all $u>2$.
\begin{equation}\label{equ:11cd}
\sum_{c+d=u}\z_\sh(1,1,c,d)=\frac12\z(2)\z(u)-\z_\sh(1,1,1,q)+2\z_*(1,1,1,q)+\z(2,1,q)+\z_*(1,2,q)+\z_*(1,1,u).
\end{equation}

Similarly, taking $\gF\big(1,0;0,0)$ in Theorem~\ref{thm:ESgenFE} and
using Theorem \ref{thm:Ma11d} we can obtain the next theorem.
\begin{thm} \label{thm:M31}
For all $v\ge 3$, let $q=v-2$ and $u=v-1$.  Set $\sum{}'=\sum_{a+b+c=v,a,b,c\in\N, a\ge2}$. Then
\begin{align*}
&\sum{}' T(a,b,c,1) =4\Big(2\z(\bar1)M_2(\breveq,\ol{e_1},\bro)+M_{24}(\breveq,\ol{e_1},\bro,\ol{e_3})-\z(\bar1)T(q,2)\Big) \\
& \hskip3cm       -6\Big(2M_{24}(\breveq,\ol{e_1},\bro,e_3)-M_2(\breveq,2\ol{e_1e_3},\bro)\Big)-2M_2(\breveq,2e_1e_3,\bro),\\
&\sum{}'
 t(a,b,c,1) =4\Big(\z(\bar1)t(q,2)+\z(\bar1)t(u,1)+t(v,1)+t(u,2)+t(q,3)-M_3(\breveu,\bro,\bar1)-\z(\bar1)M_3(\breveq,\bro,\bar1)
& \hskip3cm    -M_3(\breveq,\brt,\bar1)+ M_{34}(\breveq,\bro,\bar1,\bar1)\Big) +2t(q,1,2)+(2\z(\bar1)^2-\z(2))t(q,1),\\
&\sum{}'
 S(a,b,c,1) =\z(2)S(q,1)+4\Big(M_4(\breveq,1,\bro,1)-M_3(q,\brt,1)-M_{34}(\breveq,1,\bar1,\bar1)-M_{34}(\breveq,1,1,e_1e_2)\Big)\\
& \hskip3cm     +2\Big(T(q,2,1)-M(q,\bro,\brt)\Big)+8 M_{34}(\breveq,1,\ol{e_1e_2},1).
\end{align*}
\end{thm}

If we use $\gF(0,1;0,0)$ we may derive the restricted sum formula of the following type for all $u>3$.
\begin{equation}\label{equ:1bcd}
\sum_{b+c+d=u}\z_\sh(1,b,c,d)=\z(2,u)-\frac12\z(2)\z(u)+\z_*(1,1,u)+\z_\sh(1,1,1,q)-\z_*(1,q,2).
\end{equation}

Further, by taking partial derivatives in Theorem~\ref{thm:ESgenFE} we can find restricted sum formulas
in which only one component is equal to 2.
For convenience, for any $a_1,\dots,a_{d+d'}\in\Z$ and $1\le j_1,\dots,j_r\le d+d'$ we denote by
\begin{equation*}
    \gF_{j_1,\dots,j_r}^\ww(a_1,\dots,a_d; a_{d+1},\dots,a_{d+d'})
\end{equation*}
the identity obtained from the following steps sequentially:
\begin{enumerate}[label=(\arabic*),leftmargin=3cm]
  \item Take $\bfx=(tx_1,\dots,tx_d), \bfy=(tx_{d+1},\dots,tx_{d+d'})$;
  \item Compare the coefficients of $t^{w-d-d'}$;
  \item Apply the partial derivative $\partial^r /\partial x_{j_1}\dotsm \partial x_{j_r}$;
  \item Evaluate at $x_i=a_i$ for $i=1,2,\dots,d+d'$.
\end{enumerate}

\begin{thm} \label{thm:MMVab2_T2bc}
Using the same notation as in Theorem~\ref{thm:Ma11d}, we have
\begin{align}\label{equ:Tab2}
\sum{}' T(a,b,2)=&\, 2T(q,3)-T(q,2,1)-2M_2(\breve{q},\ol{e_1},\breve{2})-2M_2(\breve{q},2\ol{e_1},\breve{1}),\\
\sum{}' T(2,a,b)=&\, 2T(2,u)-T(q,2,1)+2M_1(2\ol{e_2},\breve{1},\breve{q})+2M_2(2\ol{e_2},\breve{q},\breve{1}),\label{equ:T2bc}\\
\sum{}' t(a,b,2)=&\, 2t(q,3)+2t(u,2)+2\z(-1)t(q,2)+t(q,2,1)-2M_3(\breve{q},\breve{2},1),\notag \\
\sum{}' t(2,b,c)=&\, 2M_1(\bar1,\breve{2},\breve{q}) +2M_2(\breve{2},\bar1,\breve{q}) ,\notag\\
\sum{}' S(a,b,2)=&\, 2S(q,3)-2\z(-1)S(q,2)-S(q,2,1)+2M_3(\breve{q},2,\ol{e_1e_2} )+2M_2(\breve{q},\ol{e_1},2), \notag \\
\sum{}' S(2,b,c)=&\, 2S(3,q)+2M_2(\breve{2}, q\ol{e_1},\breve{1})-2M_1(\bar2,\breve{1},q)-2M_1(\bar1,\breve{2},q).\notag
\end{align}
\end{thm}

\begin{proof} We may prove the corresponding formulas for the Euler sums by using $\gF_2^\ww(1,0;0)$.
Then \eqref{equ:Tab2} follows easily. Similarly, \eqref{equ:T2bc} follows from $\gF_1^\ww(0,1;0)$. For example,
\begin{equation}\label{equ:MZVab2}
\sum{}' \z(a,b,2)= \z(u,2)+\z(q,3)-\z(q,2,1),\qquad  \sum{}' \z(2,b,c)=\z(3,q)+\z(2,u).
\end{equation}
The formulas for MtV and MSV can be derived in a similar fashion.
\end{proof}

By using stuffle relations we can derive the following quickly.
\begin{thm} \label{thm:MMVa2c}
Using the same notation as above, we have
\begin{align*}
 \sum{}' T(a,2,b)=&\,\z(\bar2)S(q,1)-2\z(\bar1)S(q,2)+2S(q,3)-T(q,2,1)\\
&\,   +\big(\z(2)\z(\bar1)+5S(3)\big)T(q)+2M(\breve{q},2,1)+M(\breve{q},1,2),\\
 \sum{}' t(a,2,b)=&\,2t(q,1,2)+t(q,2,1)+2t(u,2)-4t(w)+2t(q,3)+3\z(\bar2)t(q,1)\\
&\,  -(3\z(2)\z(\bar1)+t(3))t(q)-4\z(\bar1)t(v)+2S(v,1)-T(q,1,2),\\
 \sum{}' S(a,2,b)=&\,3\z(\bar2)T(q,1)+2\z(\bar1)T(q,2)+4T(q,3)-S(q,2,1)\\
&\, +2M(q,\brt,\bro)+M(q,1,\brt)-4M_2(q,\bar1,\brt).
\end{align*}
\end{thm}
\begin{proof} We first consider the restricted formula for MZVs. By stuffle relations, we see that
\begin{equation*}
\z(2)\z(a,b)=\z(2,a,b)+\z(a,2,b)+\z(a,b,2)+\z(a+2,b)+\z(a,b+2).
\end{equation*}
Taking the sum when $(a,b)$ ranges over pairs of positive integers such that $a\ge 2$ and $a+b=u$,
we obtain from \eqref{equ:MZVab2} and sum formula  \eqref{equ:MZVsumCFormula} (with $d=2$)
\begin{equation*}
\sum{}' \z(a,2,b)=\z(u,2)+\z(3,q)-\z(2,q,1).
\end{equation*}
Using the same ideas we can find similar restricted sum formulas for other depth 3 Euler sums
with any fixed alternating sign pattern. Then the theorem follows from the definitions of the different
MMVs in a straightforward manner.
\end{proof}

\begin{thm}\label{thm:Ma1cd}
For all $w\ge 4$, set $q=w-3$ and $u=w-2$. Then we have
\begin{align*}
&\sum{}'
T(a,1,c,d)=2\Big(S(q,2,1)+T(q,2,1)-2T(q,3)\Big)+6\Big(\z(\bar3)T(q)-M(q,\bro,\brt)-M(q,\brt,\bro)\Big)\\
& +3\z(2)\Big(S(q,1)-2\z(\bar1)T(q)\Big)+4\z(\bar1)\Big(2M_2(\breveq,\bar2,\bro)-T(q,2)\Big)-8M_{24}(\breveq,\bar2,\bro,\bar3)
+12M_3(q,\bro,\bar1,\bro),\\
&\sum{}'t(a,1,c,d)=10S(3)T(q)+4\Big(t(v,1)-t(q,3)-t(u,2)-M_2(\breveu,\bar1,1)+M_{23}(\breveq,\bar1,\bar1,\bro)\Big)\\
&+4\z(\bar1)\Big(2M_2(\breveu,\bar1)-t(q,2)\Big)+2\z(2)\Big(\z(\bar1)T(q)-M_2(\breveq,\bar1)\Big)
+2M(\breveq,\brt,1) -8\z(\bar1,\bro)t(u)-16M_1(\bar1,\brevev),\\
&\sum{}'S(a,1,c,d)=2\Big(T(q,1,2)-\z(2)M_1(q,\bro)+S(q,2,1)-T(q,2,1)\Big)-t(q,1)\Big(4\z(\bar1,\bar1)+\z(2)\Big)\\
&\hskip3cm  -4\z(\bar1)\Big(T(q,2)-T(q,1,1)-M_2(\breveq,e_3,\bro)+2M_2(\breveq,e_2,\bro)\Big)\\
&\hskip3cm  -4M_{34}(\breveq,\bro,\bar2,\bar1)
+8\Big(M_{34}(\breveq,1,\bar1,\bar1)-M_{34}(\breveq,1,2\ol{e_2},e_2e_2)\Big).
\end{align*}
\end{thm}
\begin{proof}
Denote by $\gF_\sh^\ww$ (resp. $\gF_\ast^\ww$) the weight $w$ part of the shuffle (resp. stuffle) side of $\gF$. Then we have the following terms appearing:
\begin{align*}
\gF_\sh^\ww(1,0;0,1):& F_\sh^\ww(1,0,0,1;\cdots)+F_\sh^\ww(1,0,1,1;\cdots)+F_\sh^\ww(1,2,1,1;\cdots)\\
    &   +F_\sh^\ww(1,0,1,0;\cdots)+F_\sh^\ww(1,2,1,0;\cdots)+F_\sh^\ww(1,2,1,0;\cdots),\\
\gF_\sh^\ww(0,1,0;1):& F_\sh^\ww(1,2,1,1;\cdots)+F_\sh^\ww(1,2,1,0;\cdots)+F_\sh^\ww(1,2,1,0;\cdots)+F_\sh^\ww(1,0,1,0;\cdots),\\
\gF_\ast^\ww(1,0;0,1):& F_\ast^\ww(1,0,0,1;\cdots)+F_\ast^\ww(1,0,0,1;\cdots)+F_\ast^\ww(0,1,1,0;\cdots)\\
    &   +F_\ast^\ww(0,1,1,0;\cdots)+F_\ast^\ww(0,1,0,1;\cdots)+F_\ast^\ww(1,0,1,0;\cdots),\\
\gF_\ast^\ww(0,1,0;1):& F_\ast^\ww(0,1,0,1;\cdots)+F_\ast^\ww(0,1,1,0;\cdots)+F_\ast^\ww(0,1,1,0;\cdots)+F_\ast^\ww(1,0,1,0;\cdots),
\end{align*}
where $\cdots$ could be different alternating signs. Taking the difference $\gF^\ww(1,0;0,1)-\gF^\ww(0,1,0;1)$ we immediately
find a formula for $F_\sh^\ww(1,0,1,1;\cdots)$, namely, the restricted sum formula of the form
\begin{equation*}
\sum_{a+c+d=w-1,\ a,c,d\in\N} \z(a,1,c,d;\bfmu), \quad \forall \bfmu\in\{\pm1\}^4.
\end{equation*}
We can then remove all terms with $a=1$ by using $\gF(0,0,1;0)$ (see, for example, \eqref{equ:11cd}).
These restricted sum formulas of Euler sums lead to the ones in the theorem immediately.
\end{proof}

\begin{thm}\label{thm:Mab1d}
For all $w\ge 4$, set $u=w-2$, $q=w-3$, and $\sum{}'=\sum{}_{a+b+d=w-1,a\ge 2}$. Then we have
\begin{align*}
&\sum{}' T(a,b,1,d)=\Big(2S(3)-\frac43\z(\bar1)^3\Big)T(q)+2\z(\bar1)^2S(q,1)+4\Big(M_34(\breveq,1,\bar1,e_1e_2)-\z(\bar1)M_3(\breveq,1,\bar1)\Big),\\
&\sum{}'
t(a,b,1,d)=4\Big(S(v,1)+t(q,3)
    -M_3(\breveu,1,\ol{e_2})-M_3(\breveq,\brt,\bar1)-M_2(\breveq,\bro,\bar2)+M_{24}(\breveq,\bar1,\bro,\bar1)\Big)\\
& \hskip.6cm
4\z(\bar1)\Big(S(u,1)+t(q,2)-\z(\bar1)t(u)-M_2(\breveq,\bar1,\bro)\Big)-6S(3)t(q)+8\Big(S(2)t(u)-t(w)-\z(\bar1)t(v)\Big)   ,\\
&\sum{}'
S(a,b,1,d)=4\Big(T(q,3)-S(2)T(q,1)+M(q,\bro,\brt)-M_1(\barq,1,\brt)\Big) \\
& \hskip3cm     -8\Big(M_{12}(q,\ol{e_3},1,\bro)+M_2(\breveq,2\ol{e_1},\bro)\Big)+16M_{123}(\breveq,\bar1,\bar1,\bar1).
\end{align*}
\end{thm}
\begin{proof}
Using Theorem \ref{thm:Ma1cd} we see that the identity $\gF(1,0,1;0)$ in Theorem~\ref{thm:ESgenFE} produces a weighted sum
formula that includes non-admissible terms. To remove these terms we may use the weighted sum
formula produced by $\gF(0,1,1;0)$.
\end{proof}

As a corollary of the above restricted sum formulas we can now derive the sum formula of Euler sums of all possible
alternating sign patterns and all the MMVs in the classical sense, which generalizes the
well-known sum formula \eqref{equ:MZVsumCFormula}. In depth 4 we have the following results explicitly.
We only write down the formulas for Euler sums for two types of alternating sign patterns since the others are
a little more complicated.
\begin{thm}\label{thm:M4}
For all $w\ge 5$, set $v=w-1$, $u=w-2$, $q=w-3$, and $\sum{}'=\sum{}_{a+b+c+d=w,a\ge 2}$.  Then 
\begin{align*}
&\sum{}'\z(a,b,c,\bard)=\z(\barw)+M_1(\baru,\brt)+M_1(\barv,\bro)+M_1(\barq,\breve3)+M_{12}(\barq,2,\bro)+M_{12}(\barq,1,\brt)\\
&\hskip3cm  +M_{12}(\baru,1,\bro)+M_{123}(\barq,1,1,\bro),\\
&\sum{}'\z(a,\barb,\barc,\bard)= M_{123}(\barq,\bar1,\bar1,\bro)-M_{12}(\barq,\ol{e_3},\brt)
+M_{123}(\bar1,\bar1,\barq,\bro) -M_{12}(\bar1,\ol{e_3},\breveu)-\z(\bar2,u)-\z(\baru,2),\\
&\sum{}'T(a,b,c,d)=-2\z(\bar3)T(q)+4\Big(\z(\bar1)T(q,2)+M_2(\breveq,2\ol{e_2e_3},\bro)\Big)
+8\Big(M_{24}(\breveq,\bar2,\bro,e_3)-\z(\bar1)M_2(\breveq,\bar2,\bro)\Big),\\
&\sum{}'t(a,b,c,d)=\z(2)\Big(t(u)+t(q)\z(\bar1)
+M_1(\bar1,\breveq)-M_2(\breveq,\bar1)\Big)-4M_1(\bar1,\breveq,2)+8M_{123}(\bar1,\bar1,\bar1,\breveq),\\
&\sum{}'S(a,b,c,d)=2\z(2)M_1(\barq,\bro)+4M_1(\barq,1,\brt)+8M_{12}(\barq,e_3,\bro,\bro).
\end{align*}
\end{thm}

\section{Restricted weighted sum formulas: depth two and three}
We may obtain some variations of the (restricted) sum formulas by taking partial derivatives after applying Theorem~\ref{thm:ESgenFE}. In \cite{GuoLeZh2014}, Guo, Lei and the author studied some families of weighted sums
formulas of MZVs in which the weights (i.e. the coefficients) of the MZVs are given
by polynomials of the arguments. The corresponding formulas for depth two Euler sums
have played an important role in the proof of the Kaneko-Tsumura conjecture and other identities \cite{BCJXXZ2020}.
We also derived many other identities such as the following (see \cite[Theorem 4.4]{BCJXXZ2020}):
\begin{equation}\label{equ:z1bc}
\begin{split}
\sum_{b+c=w-1}\z_*(1,b,c)=&\,\z(2,w-2)+\z_*(1,w-1)+\z_*(1,1,w-2)-\z(2)\z(w-2),\\
\sum{}\z_\sha(1,b,c)=&\,\z(2,w-2)+\z_*(1,w-1)+\z_*(1,1,w-2)-\frac12\z(2)\z(w-2).
\end{split}
\end{equation}

We now derive more identities in depth three. First, we have the following restricted weighted sum formulas.
\begin{thm}\label{thm:MMVab1Wa_ab1Wab}
For all $w\ge 4$, set $v=w-1$, $u=w-2$, $q=w-3$, $\tla=a-1,\tlb=b-1$, and
$\sum{}'=\sum{}_{a+b=w-1,a\ge 2}$.  Then we have
\begin{align*}
\sum{}'\tla\, T(a,b,1)=&\,T(2)T(q,1)-2qT(u,1_2)-2T(q,2,1)-T(q,1,2),\\
\sum{}'\tla\,  t(a,b,1)=&\,\frac13t(2)t(q,1)-2qM_3(\breveu,\bro,1)-C(\breveq),\\
\sum{}'\tla\,  S(a,b,1)=&\,3S(2)S(q,1)-2qM_3(u,\bro,1)-C(q),\\
\sum{}'\tlb\,  T(a,b,1)=&\,2q\Big(T(u,2)+M(u,\bro,\bro)\Big)+2T(q,2,1)+T(q,1,2)-T(2)T(q,1),\\
\sum{}'\tlb\,  t(a,b,1)=&\, 2q\Big(t(u,1_2)+t(v,1)+t(u,2)+t(u,1)\z(\bar1)\Big)-\frac13 t(2)t(q,1)+C(\breveq),\\
\sum{}'\tlb\,  S(a,b,1)=&\,2q\Big(S(u,2)+M(\breveu,1_2)-\z(\bar1)S(u,1)\Big)-3S(2)S(q,1)+C(q).
\end{align*}
where $C(x)=2M_3(x,\brt,1)+M(x,\bro,\brt)$ for $x=q$ or $x=\breveq$.
\end{thm}

\begin{proof} First, $\gF_3^\ww(1,0;0)$ yields the restricted weighted
sum formulas for Euler sums of the form $\sum \tla\, \z_\sh(a,b,1;\bfmu)$ where $a+b=w-1$.
This may seem to contain some non-admissible terms with $a=1$. But when $a=1$ the weight $\tla=a-1=0$.
So we can obtain the first three formulas in the theorem immediately.

Similarly, we may obtain the  restricted weighted
sum formulas of the form $\sum \tlb \z_\sh(a,b,1;\bfmu)$ by using $\gF_1^\ww(1,0;0)$.
However, non-admissible terms $\z_\sh(1,u,1;\bfmu)$ appear nontrivially this time,
which can be removed easily. This leads to the last three
restricted weighted sum formulas in the theorem.
\end{proof}

\begin{thm}\label{thm:Ta1cWac}
For all $w\ge 4$, set $v=w-1$, $u=w-2$, $q=w-3$, $\tla=a-1,\tlc=c-1$, and
$\sum{}'=\sum{}_{a+c=w-1,a\ge 2}$.  Then we have
\begin{align*}
\sum{}' \tla\, T(a,1,c)=&\,\z(3)T(q)-2qS(u,2)-4S(q,3),\\
 \sum{}' \tla\, t(a,1,c)=&\,4uM_2(\brevev,1)-\z(2)t(u)+C(\breveq)-2q M_2(\breveu,\bar1,\bro)+A_w,\\
 \sum{}' \tla\, S(a,1,c)=&\,T(3)S(q)-2\z(\bar1)T(q,2)-2T(q,3)+C(q)-2q M_2(u,\bar1,\bro),\\
  \sum{}' \tlc\, T(a,1,c)=&\,2q\Big(S(u,2)+M_3(\breveu,1,\bar1)+\z(\bar1)^2T(u)-\z(\bar1)S(u,1)\Big)+4\Big(S(q,3)-S(3)T(q)\Big),\\
 \sum{}' \tlc\, t(a,1,c)=&\,2\Big(u\,\z(\bar2)t(u)-qt(v,1)\Big)-4q\Big(t(w)+\z(\bar1)t(v)\Big)-4M_2(\brevev,1)-C(\breveq)-A_w,\\
 \sum{}' \tlc\, S(a,1,c)=&\,2q\Big(T(u,2)+\z(\bar1)T(u,1)\Big)+2T(q,3)+2\z(\bar1)T(q,2)-T(3)S(q)-C(q)-D(u).
\end{align*}
where $A_w=t(3)t(q)+2\Big(S(u,2)-t(q,3)-\z(\bar1)t(q,2)\Big)$, $C(x)=2M_3(x,\brt,\ol{e_2})$, and $D(x)=2qM_3(x,\bro,\bar1)$.
\end{thm}

\begin{proof}
We first prove the last three restricted weighted sum formulas in the theorem using $\gF_2^\ww(1,1;0)$.
This yields the restricted weighted sum formulas for Euler sums of the form $\sum \tlc\, \z_\sh(a,1,c;\bfmu)$ where $a+c=w-1$.
We then remove all the terms $q\z_\sh(1,1,u;\bfmu)$.

To prove the first three formulas, by the stuffle relation we see that
\begin{equation*}
\z_*\genfrac{(}{)}{0pt}{0}{a\, ,\, 1\, ,\, c}{\mu_1,\mu_2,\mu_3}=\z_*\genfrac{(}{)}{0pt}{0}{1}{\mu_2}\z_*\genfrac{(}{)}{0pt}{0}{a\, ,\, c}{\mu_1,\mu_3}
-\z_*\genfrac{(}{)}{0pt}{0}{1\, ,\, a\, ,\, c}{\mu_2,\mu_1,\mu_3}-\z_*\genfrac{(}{)}{0pt}{0}{a\, ,\, c\, ,\, 1}{\mu_1,\mu_3,\mu_2}
-\z_*\genfrac{(}{)}{0pt}{0}{a+1,c}{\mu_1\mu_2,\mu_3}-\z_*\genfrac{(}{)}{0pt}{0}{a,c+1}{\mu_1,\mu_2\mu_3}.
\end{equation*}
Multiply $w-3=a+c-2=\tla+\tlc$ on both sides of the equation and take the sum with $a+c=w-1$.
Then there are five sums on the right-hand side, denoted by $\Sigma_j$, $1\le j\le 5$. Observe that
$\Sigma_1$, $\Sigma_4$  and $\Sigma_5$ are given by \cite[Theorem 3.2]{BCJXXZ2020},
$\Sigma_2$ by \eqref{equ:z1bc}, $\Sigma_3$ by \cite[Theorem 4.2]{BCJXXZ2020}.
Therefore the first three restricted weighted sum formulas in the theorem follow from this
and the last three formulas in the theorem. This completes the proof of the theorem.
\end{proof}

In general, a restricted sum formula for MZVs, MTVs, MtVs and MSVs cannot allow the first component to be 1
due to divergence. However, for Euler sums such formulas make sense if the corresponding alternating sign is $-1$.
Hence, we may get the following restricted weighted sum formulas.

\begin{thm}\label{thm:EulerSum1bcWb_c}
For all $w\ge 4$, set $v=w-1$, $u=w-2$, $q=w-3$, $\tlb=b-1,\tlc=c-1$, and
$\sum{}=\sum{}_{b+c=w-1,\, b,c\in\N}$.  Then we have
\begin{align*}
&\sum \tlb\, \z(\bar1,\barb,\barc)=C(q,u,v)-q\z(\bar1,u,\bar1)+M_{23}(\brt,\ol{e_1},q)+M_{12}(\bar1,q,\brt),\\
&\sum \tlb\, \z(\bar1,\barb,c)=C(\barq,\baru,\barv)-q\z(\bar1,\baru,1)+M_{23}(\brt,\ol{e_1},q),\\
&\sum \tlb\, \z(\bar1,b,c)=C(q,\baru,\barv)-q\z(\bar1,u,1)-M_{12}(\bar1,q,\brt),\\
&\sum \tlb\, \z(\bar1,b,\barc)=C(\barq,u,v)-q\z(\bar1,\baru,\bar1),
\end{align*}
where $C(x,y,z)=\z(\bar1,2,x)-\z(\bar2,x,\bar1)-\z(\bar2,y)+\z(\bar1,z)$,   and
\begin{align*}
&\sum \tlc\, \z(\bar1,\barb,\barc)=q\Big(\z(\bar1,1,u)+A(1,u,v)\Big)
   +B(1,u,q,v) -M_{23}(\brt,\ol{e_1},q)-M_{12}(\bar1,q,\brt),\\
&\sum \tlc\, \z(\bar1,\barb,c)=q\Big(\z(\bar1,1,\baru)+A(1,\baru,\barv)\Big)
   +B(1,\baru,\barq,\barv)-M_{23}(\brt,\ol{e_1},\barq),\\
&\sum \tlc\, \z(\bar1,b,c)=q\Big(2\z(\bar1,\bar1,u)-\z(\bar1,1,u)+A(\bar1,u,\barv)\Big)
     +B(\bar1,\baru,q,\barv)+M_{12}(\bar1,q,\brt),\\
&\sum \tlc\, \z(\bar1,b,\barc)=q\Big(2\z(\bar1,\bar1,\baru)-\z(\bar1,1,\baru)+A(\bar1,\baru,v)\Big)
    +B(\bar1,u,\barq,v),
\end{align*}
where $A(a,x,y)=\z(\bar1,u,a)+\z(\bar2,x)+\z(\bar1,y)$ and $B(a,x,y,z)=\z(\bar2,x)+\z(\bar2,y,a)-\z(\bar1,2,y)-\z(\bar1,z)$.
\end{thm}
\begin{proof} Applying $\gF_3^\ww(0,1;0)$ we obtain the first four formulas with the help of
the corresponding sum formulas of the form $\sum \z(2,b,c;\bfmu)$ used in the proof of Theorem~\ref{thm:MMVab2_T2bc}.
The last four formulas follow from the first four and the identities produced by $\gF_2^\ww(0,1;0)$.
\end{proof}

The next result shows that we have a weighted sum formula in depth 3 for all Euler sums,
MTVs, MtVs and MSVs, when the weight is a linear polynomial of the arguments $a$, $b$ and $c$.
\begin{thm}\label{thm:MMV3Wa_3Wb_3Wc}
For all $w\ge 4$, set $u=w-2$, $q=w-3$, $\bfs=(a,b,c)$, $\tla=a-1,\tlb=b-1,\tlc=c-1$, and
$\sum{}'=\sum{}_{|\bfs|=w,a\ge 2}$.  Then we have
\begin{align*}
\sum{}' \tla\, T(\bfs)=&\, T(q,2,1)+qT(u,1,1)+T(q,1,2)+T(3)T(q)-T(2)T(q,1),\\
\sum{}' \tla\, t(\bfs)=&\, 2S(3)t(q)+qM(\breveu,1,1)+M(\breveq,2,1)+M(\breveq,1,2)-S(2)S(q,1),\\
 \sum{}' \tla\, S(\bfs)=&\, t(3)S(q)+qM(u,1,\bro)+M(q,2,\bro)+M(q,1,\brt)-3S(2)t(q,1),\\
 \sum{}' \tlb\, T(\bfs)=&\, 3\z(2)M_1(\barq,e_1)-2T(q,3)-\big(3\z(2)\z(\bar1)+T(3)\big)T(q)\\
  &\, -2q\Big(M(u,\bro,\bro)+T(u,2)\Big)-2\Big(M_2(\breveq,2e_1,\bro)+M_2(\breveq,e_1,\brt)\Big),\\
 \sum{}' \tlb\, t(\bfs)=&\, 4M_{12}(\bar1,2,\breveq)-2q\Big(S(u,2)+M_1(\bar1,\breveu,1)+M(\breveu,1,1)\Big)
    +2\Big(M_2(\brt,\ol{e_1},\breveq)+M_1(2,\breveq,1)\Big),\\
 \sum{}' \tlb\, S(\bfs)=&\, 4\z(\bar1)M_1(q,\brt)+2q\Big(t(u,1)\z(\bar1)-M_2(\breveu,e_3,\bro)\Big)
  +2\Big(T(2)M_1(q,\bro)-M_1(q,\brt,1)-M_1(q,1,\brt)\Big),\\
 \sum{}' \tlc\, T(\bfs)=&\, q\Big(\z(2)T(u)+T(u,1,1)\Big)-T(q,2,1)+T(2,1,q)\\
  &\, +2\Big(T(q,3)-T(3)T(q)+M_2(\breveq,2e_1,\bro)-M_2(\breveq,\ol{e_1},\brt)\Big),\\
 \sum{}' \tlc\, t(\bfs)=&\,  q\Big(4M_{12}(\bar1,\bar1,\breveu)+2M_1(\bar1,\breveu,1)+M(\breveu,1,1)+\z(2)T(u)\Big)-T(2,1,q) -M(\breveq,2,1)\\
  &\,
+2\Big(S(q,3)-M_{12}(2,q,1)+M_{12}(2,\barq,1)-M_{12}(2,\bar1,\breveq)-M_{12}(\bar2,\bar1,\breveq)\Big)-4M_{12}(\bar1,2,\breveq),\\
 \sum{} \tlc\, S(\bfs)=&\, qS(u,1,1)-2q\z(\bar1)T(u,1)+T(q,1,2)+T(q,2,1)\\
  &\, -T(3)S(q)-T(2)T(q,1)+4M_{13}(q,\brt,\ol{e_2})-4\z(\bar1)M_1(q,\brt).
\end{align*}
\end{thm}
\begin{proof} For any $\bfmu\in\{\pm1\}^3$, applying $\gF_1^\ww\big((0,0;\mu_1,\mu_2);(1,\mu_1\mu_2\mu_3)\big)$ and using
Theorem~\ref{thm:MMVab1Wa_ab1Wab} we obtain the weighted sum formulas for Euler sums of the form
$\sum \tla\, \z_\sh(\bfs;\bfmu)$ where $\bfs\in\N^3$ and $|\bfs|=w$.
This may seem to contain some non-admissible terms with $a=1$. But when $a=1$ the weight $\tla=a-1=0$.
Then we can obtain the first three formulas in the theorem immediately.

Similarly, we may obtain the weighted
sum formulas of the form $\sum \tlb \z_\sh(\bfs;\bfmu)$ by $\gF_2^\ww(0,0;1)$.
To remove the terms with $a=1$ we need to use $\gF_3(0,1;0)$ which leads to the
the restricted weighted sum formulas of the form $\sum \tlb \z_\sh(1,b,c;\bfmu)$.
After removing them we immediately obtain
\begin{equation*}
  \sum_{a+b+c=w,a\ge2} \tlb \z(\bfs;\bfmu)=\sum_{a+b+c=w} \tlb \z_\sh(\bfs;\bfmu)-\sum_{b+c=w-1} \tlb \z_\sh(1,b,c;\bfmu),
\end{equation*}
which imply the middle three formulas in the theorem.

Finally, to prove the last three formulas we first use $\gF_3^\ww(0,0;1)$ and then remove
the non-admissible terms by using $\gF_2^\ww(0,1;0)$. This completes the proof of the theorem.
\end{proof}

Next, we may also use quadratic polynomials in the weighted sum formula in depth 2 for all Euler sums,
MTVs, MtVs and MSVs. It is conceivable that actually this can be generalized to any polynomials in $\Z[a,b,c]$.
\begin{thm}\label{thm:MMVabWallQ}
For all $w\ge 4$, set $u=w-2$, $q=w-3$, $\bfs=(a,b,c)$, $\tla=a-1,\tlb=b-1,\tlc=c-1$, and
$\sum{}'=\sum{}_{|\bfs|=w,a\ge 2}$.  Then we have
\begin{align*}
\sum{} \tla^2  T(a,b)=&\,2T(3)T(q)-2T(q,3)-u^2T(v,1)-(2q+1)T(u,2)+T(2)T(u),\\
\sum{} \tla^2  t(a,b)=&\,2S(3)t(q)-2S(q,3)-u^2S(v,1)-(2q+1)S(u,2)+S(2)T(u),\\
\sum{} \tla^2  S(a,b)=&\,2t(3)S(q)-2t(q,3)-u^2t(v,1)-(2q+1)t(u,2)+t(2)S(u),\\
\sum{} \tla\tlb T(a,b)=&\,qT(2)T(u) - 2T(3)T(q) + 2T(q, 3) + qT(u, 2),\\
\sum{} \tla\tlb t(a,b)=&\,qS(2)t(u) - 2S(3)t(q) + 2S(q, 3) + qS(u, 2),\\
\sum{} \tla\tlb S(a,b)=&\,qt(2)S(u) - 2t(3)S(q) + 2t(q, 3) + qt(u, 2),\\
\sum{}' \tlb^2  T(a,b)=&\, 2T(3)T(q)-(2q+1)T(2)T(u)-2T(q,3)+T(u,2)+u^2\Big(S(v,1)-2\z(\bar1)T(v)\Big),\\
\sum{}' \tlb^2  t(a,b)=&\, 2S(3)t(q)-(2q+1)S(2)t(u)-2S(q,3)+S(u,2)+u^2\Big(t(v,1)+2\z(\bar1)t(v)+2t(w)\Big),\\
\sum{}' \tlb^2  S(a,b)=&\, 2t(3)S(q)-(2q+1)t(2)S(u)-2t(q,3)+t(u,2)+u^2T(v,1)
\end{align*}
\end{thm}
\begin{proof} Let $\gF_{\sh;2,2}^\ww$ (resp. $\gF_{\ast;2,2}^\ww$) be the shuffle (resp. stuffle) side
of $\gF_{2,2}^\ww$. For any $\mu_1,\mu_2=\pm1$, applying the second partial derivative, we get
\begin{align*}
\gF_{\sh;2,2}^\ww\Big((x;\mu_1\mu_2);(y;\mu_1)\Big)=&\,
\frac{\partial^2}{\partial y^2} (F_\sh(x+y,y;\mu_1\mu_2,\mu_2)+F_\sh(x+y,x;\mu_1,\mu_2)),\\
\gF_{\ast;2,2}^\ww\Big((x;\mu_1\mu_2);(y;\mu_1)\Big)=&\,
\frac{\partial^2}{\partial y^2} (F_\ast(x,y;\mu_1\mu_2,\mu_2)+F_\ast(y,x;\mu_1,\mu_2)).
\end{align*}
Then evaluating at $x=1,y=0$ we obtain the weighted
sum formulas for Euler sums of the form $\sum \tla(\tla-1)\, \z_\sh(\bfs;\bfmu)$ where $\bfs\in\N^3$ and $|\bfs|=w$.
This may seem to contain some non-admissible terms with $a=1$. But, as before, when $a=1$ the weight $\tla(\tla-1)=0$.
Thus, we arrive at the first three formulas in the theorem by adding the sum formula $\sum \tla\, \z_\sh(\bfs;\bfmu)$
obtained in \cite[Corollary 5.2]{BCJXXZ2020}. The middle three formulas follow from $\gF_{1,2}^\ww(1;0)$ while
the last three follow from $\gF_{1,1}^\ww(1;0)$. This completes the proof of the theorem.
\end{proof}

\section{Weighted sum formulas: depth 3 and 4}
Previously, the author and his collaborators studied a few (restricted and/or weighted) sum formulas of Euler sums of depths 2 and 3. In this section, we explore some more such formulas in depths 3 and 4.

Fix depth $d\in\N$ and weight $w\ge d+1$. We consider MZVs first.
In Theorem~\ref{thm:ESgenFE}, taking $\bfx=\big( 1_{d-1}\big)$ and $\bfy=(-1)$ we get the following terms
 on the shuffle side:
\begin{equation*}
 S_\sh^\ww \big(0_{d-1},-1\big)+\sum_{k=1}^{d-1} S_\sh^\ww  \big(0_{k}, 1_{d-k}\big)
\end{equation*}
which can all be found out explicitly using the results in section~\ref{sec:restSumForm}. On the stuffle side,
we get the following terms
\begin{align*}
&\,\sum_{k=0}^{d-1} S_\st^\ww \big(1_{k},-1,1_{d-k-1} \big)
 +\sum_{k=0}^{d-2} S_\st^\ww  \big(1_{k}, (-1)\ot(1), 1_{d-k-2}\big)\\
=&\,-\sum_{\bfs\in\N^{d}}\Big( (-1)^{s_1}+\dotsm+(-1)^{s_d} \Big) \zeta_\st(\bfs)
    +\sum_{k=1}^{d-1}\sum_{\bfs\in\N^{d-1}} \frac{1+(-1)^{s_k}}2 \zeta_\st(\bfs) \\
=&\,-\sum_{\bfs\in\N^{d}}\Big( (-1)^{s_1}+\dotsm+(-1)^{s_d} \Big) \zeta_\st(\bfs)
    +\frac{d-1}2\sum_{\bfs\in\N^{d-1}}\zeta_\st(\bfs)
+\frac12 \sum_{\bfs\in\N^{d-1}}\Big( (-1)^{s_1}+\dotsm+(-1)^{s_{d-1}} \Big) \zeta_\st(\bfs).
\end{align*}
Thus by recursive computation we can obtain the weighted sum formula for
\begin{equation*}
A_d(w):=\sum_{\bfs\in\N^d}\Big( (-1)^{s_1}+\dotsm+(-1)^{s_d} \Big) \zeta_\st(\bfs).
\end{equation*}

\begin{thm} \label{thm:zdW_1}
Notation as before. Then we have
\begin{align*}
A_2(w)=&
\left\{
  \begin{array}{ll}
    0 , \quad & \hbox{if $2\nmid w$;} \\
   \z(w)-2\z_\st(1,v),\quad & \hbox{if $2|w$,}
  \end{array}
\right.  \\
A_3(w)
=& \left\{
     \begin{array}{ll}
       \z(w)-3\z_*(1,1,u)-\z(2,u) ,\quad & \hbox{if $2\nmid w$;} \\
     \frac32\z(w)-\z_*(1,1,u)-\z(2,u)-\z_\st(1,v), \quad & \hbox{if $2|w$,}
     \end{array}
   \right.\\
A_4(w)
=&c(w)+\left\{
    \begin{array}{ll}
   2\z(w)+\frac32\z_*(1,v)-2\z_*(1,1,u)-2\z_*(1_3,q), & \hbox{if $2\nmid w$;} \\
  \frac94\z(w)+\z_*(1,v)-\z_*(1,1,u)-4\z_*(1_3,q), & \hbox{if $2|w$.}
    \end{array}
  \right.
\end{align*}
where $c(w)=\z_*(1,q,2)-\z_*(1,2,q)-\z(2,1,q)-\frac32\z_*(u)\z(2)$.
\end{thm}
\begin{proof}
When $d=2$, the identity produced by $\gF(1;-1)$ yields
\begin{equation*}
A_2(w):=\sum_{\bfs\in\N^2}\Big( (-1)^{s_1}+(-1)^{s_2} \Big) \zeta_\st(\bfs)=(1+(-1)^w)\Big(\frac12\z(w)-\z_\st(1,v)\Big).
\end{equation*}
This is consistent with the well-known restricted sum formula of double zeta values when $w$ is
even (cf. \cite[Theorem 1]{GanglKaZa2006}). When $d=3$, from $\gF(1,1;-1)$ we have
\begin{equation*}
\frac{(-1)^w}2 \z(2)\z(u)+(1-(-1)^w) \z_\sh(1,1,u) +\sum \z_\sh(1,a,b)
=- A_3(w)+\z_\st(1,v)+\z(w)+\frac12 A_2(w).
\end{equation*}
By \eqref{equ:z1bc}
\begin{equation*}
\sum{}\z_\sh(1,a,b)=\z(2,u)+\z_*(1,v)+\z_\sh(1,1,u)-\z(2)\z(u).
\end{equation*}
Therefore
\begin{multline*}
A_3(w):=\sum_{\bfs\in\N^3}\Big( (-1)^{s_1}+(-1)^{s_2}+(-1)^{s_3} \Big) \zeta_\st(\bfs) \\
=((-1)^w-2) \z_*(1,1,u)-\z(2,u)
+\z(w)+\frac12(1+(-1)^w)\Big(\frac12\z(w)-\z_\st(1,v)\Big).
\end{multline*}
Similarly, when $d=4$ from $\gF(1,1,1;-1)$ we get
\begin{multline*}
-\sum_{a+b=u} (-1)^{b-1} \z_\sh(1,1,a)\z_\sh(b)+(1+(-1)^w)\z_\sh(1_3,q)+\sum_{c+d=u} \z_\sh(1_2,c,d)+\sum_{b+c+d=v} \z_\sh(1,b,c,d)\\
=-\sum_{a+b=u} (-1)^{b-1}\z_*(1,1,a)\z_*(b)-A_4(w)
+\frac32\sum_{\bfs\in\N^3}\zeta_\st(\bfs)+\frac12 A_3(w).
\end{multline*}
Using \eqref{equ:11cd}, \eqref{equ:1bcd} and Example~\ref{eg:sh2stu}, we see that
\begin{align*}
&\, \sum_{a+b=u} (-1)^{b-1} (\z_\sh(1,1,a)-\z_*(1,1,a))\z_*(b) \\
=&\,\frac12\z(2)\sum_{a+b=u} (-1)^{b-1}\z_*(a)\z_*(b)-\frac{1}3(-1)^w\z(3)\z(q) \\
=&\,-\frac12\z(2)A_2(u)+\frac14(1+(-1)^w)\z(2)\z(u)-\frac{1}3(-1)^w\z(3)\z(q)\\
=&\,\frac12(1+(-1)^w)\z(2)\z_\st(1,q)-\frac{1}3(-1)^w\z(3)\z(q).
\end{align*}
Therefore
\begin{multline*}
A_4(w)=
(1+(-1)^w)\Big(\frac18\z(w)-\frac14\z_*(1,v)+\frac12\z_*(1,1,u)-\z_*(1_3,q)\Big)\\
+2\z(w)+\frac32\z_*(1,v)-2\z_*(1,1,u)-2\z_*(1_3,q)
+\z_*(1,q,2)-\z_*(1,2,q)-\z(2,1,q)-\frac32\z_*(u)\z(2).
\end{multline*}
This completes the proof of the theorem.
\end{proof}

We can now derive a result first proved by Shen and Cai in \cite[Theorem 3]{ShenCai2012}.
\begin{cor}
Notation as before. For any even $w\ge4$ we have
\begin{equation*}
\sum_{|\bfs|=w,\ s_1\ge 2}\Big( (-1)^{s_1}+ (-1)^{s_1}+(-1)^{s_3} \Big) \z(\bfs)=\frac32\z(w)-\z(2)\z(u).
\end{equation*}
\end{cor}
\begin{proof} From the expression of $A_3(w)$ we see that
\begin{align*}
&\, \sum_{|\bfs|=w,\ s_1\ge 2}\Big( (-1)^{s_1}+ (-1)^{s_1}+(-1)^{s_3} \Big) \z(\bfs)\\
=&\, A_3(w)- \sum_{b+c=v}\Big( -1+ (-1)^b+(-1)^c \Big) \z(1,b,c)\\
=&\, A_3(w)+\sum_{b+c=v}\z(1,b,c)     \hskip1cm \text{(since $b+c$ is odd)}\\
=&\, \frac32\z(w)-\z(2)\z(u)
\end{align*}
by \eqref{equ:z1bc}. The corollary is thus proved.
\end{proof}

When experimenting with $\bfy=-1$ we find the following weighted restricted sum formulas.
\begin{thm} \label{thm:za1cW_1}
Notation as before. If weight $w\ge 4$ is even then we have
\begin{equation*}
\sum_{a+c=w-1,a\ge 2}  (-1)^a \z(a,1,c)= \z(u,2)-\z(v,1)+\frac12 \z(w).
\end{equation*}
\end{thm}
\begin{proof}
Let $\tla=a-1$ as before. Taking $\gF(0,1;-1)+\gF(1,0;-1)$ we find that
\begin{align*}
 & \z_\sh(1,1)\z(u)+2\sum_{a+c=v}  (-1)^{\tla}  \z_\sh(a,1,c)+\sum_{b+c=v} (-1)^{b+c} \z_\sh(1,b,c)+\sum_{a+c=v} (-1)^{a+c} \z_\sh(a,1,c) \\
=& \z_*(1,1)\z(u)+\sum_{a+b=w,a\ge2} (-1)^a \z_*(a,b)+\sum_{a+b=w,b\ge2} (-1)^b \z_*(a,b)
\end{align*}
since all the other terms cancel due to evenness of $w$. Noticing that $v$ is odd, by
\eqref{equ:z1bc} (sum formula of $\z_\sh(1,b,c)$), \cite[Theorem 4.8]{BCJXXZ2020} (sum formula of $\z_*(a,1,c)$), first
equation of Example \ref{eg:sh2stu} (conversion from $\z_\sh$ to $\z_*$), and the formula of $A_2(w)$ given by Theorem~\ref{thm:zdW_1},
we get
\begin{equation*}
2\sum_{a+c=w}  (-1)^{\tla}  \z_\sh(a,1,c)=2\z_\sh(1,1,u)+2\z(v,1)-2\z(u,2)-\z(w).
\end{equation*}
The theorem follows immediately.
\end{proof}

\begin{cor}\label{cor:a1cEvenOdd}
Notation as before. If weight $w\ge 4$ is even then we have
\begin{align*}
\sum_{\substack{a+c=w-1,a\ge 2\\ a:\, \text{even},\ c:\, \text{odd}}}  \z(a,1,c) &\, = \frac12\z(2)\z(u)-\frac14 \z(w),\\
\sum_{\substack{a+c=w-1,a\ge 2\\ a:\, \text{odd},\ c:\, \text{even}}}  \z(a,1,c) &\, =\z(v,1)+\frac12\Big(\z(2,u)-\z(u,2)\Big)-\frac14 \z(w).
\end{align*}
\end{cor}

\begin{proof}
This follows from Theorem~\ref{thm:za1cW_1} and the restricted sum formula
\begin{equation*}
\sum_{a+c=w,a\ge 2} \z(a,1,c)=\z(v,1)+\z(2,u)
\end{equation*}
by  \cite[Theorem 4.8]{BCJXXZ2020}.
\end{proof}

\begin{thm} \label{thm:zabcW_11_1}
Notation as before. If weight $w\ge 4$ is even then we have
\begin{equation*}
\sum_{a+b+c=w,a\ge 2}  [(-1)^c-(-1)^a] \z(a,b,c)=\frac12 \z(w)- \z(2)\z(u).
\end{equation*}
\end{thm}
\begin{proof}
Taking $\gF(-1,1;0)$ we find that
\begin{align*}
 & \sum_{a+b=v}  (-1)^{\tla}  \z_\sh(a,b,1)+\sum_{a+b+c=w} (-1)^{\tla} \z_\sh(a,b,c)+\sum_{a+b+c=w} (-1)^{a+b} \z_\sh(a,b,c) \\
=& \sum_{a+b=v} (-1)^{\tla} \Big( \z_*(a,b,1)+\z_*(a,1,b)+ \z_*(1,a,b)\Big)+
 \sum_{a+b=w,a\ge2} (-1)^a \z_*(a,b)+\sum_{a+b=w,b\ge2} (-1)^{a-1} \z_*(a,b).
\end{align*}
Noticing that $\z_\sh(a,b,1)= \z_*(a,b,1)$ since $a+b=v\ge 3$,  by Theorem~\ref{thm:za1cW_1} we
see that
\begin{align*}
 & \sum_{a+b+c=w,a\ge2} (-1)^{\tla} \z_\sh(a,b,c)+\sum_{a+b+c=w,a\ge 2} (-1)^{a+b} \z_\sh(a,b,c)+\Big(\z_\sh(1,1,u)-\z_*(1,1,u)\Big) \\
=& \z_*(1,1,u)+\z(v,1)-\z(u,2)-\frac12 \z(w)+\z_*(1,v)-\z(v,1)-\sum_{b+c=v,a\ge2}  \z_\sh(1,b,c).
\end{align*}
Hence the theorem follows from  Example \ref{eg:sh2stu} and \eqref{equ:z1bc} immediately.
\end{proof}

Returning to the general Euler sums, sometimes we can derive weighted sum formulas for MTVs
even though the corresponding result cannot be found
for Euler sums in general. The key is Proposition~\ref{prop:Tsign}. For example,
for any $\bfmu\in\{\pm1\}^{d-1},\nu\in\{\pm1\}$, taking $\bfx=\big(0_k,1,0_{d-k-2};\bfmu\big)$,
$\bfy=(1;\nu)$ in Theorem~\ref{thm:ESgenFE},
and setting $\xi_0=1$, $\xi_j=\mu_1\dotsm\mu_j$, $1\le j<d$, we find
\begin{align*}
&\,\ \ \sum_{j=0}^{k-1} S_\sh^\ww \big(1_{j+1},0_{k-j},1,0_{d-k-2};\cdots \big)\\
&\, +\sum_{j=2}^{d-k-1} S_\sh^\ww  \big(1_k,2,1_{j},0_{d-k-j-1};\mu_1,\dotsc,\mu_{k+j},\nu \xi_{k+j},\nu \xi_{k+j+1},\mu_{k+j+2},\dotsc,\mu_{d-1}\big)\\
&\,+ S_\sh^\ww  \big(1_k,2,1,0_{d-k-2};\mu_1,\dotsc,\mu_k,\  \nu\xi_k\,  ,\nu\xi_{k+1},\, \mu_{k+2}\, ,\mu_{k+3},\dotsc,\mu_{d-1} \big)\\
&\,+ S_\sh^\ww  \big(1_k,2,1,0_{d-k-2};\mu_1,\dotsc,\mu_k,\mu_{k+1},\nu\xi_{k+1},\nu\xi_{k+2},\mu_{k+3},\dotsc,\mu_{d-1} \big) \\
\equiv &\, S_\st^\ww \big(0_k,(1)\ot(1),0_{d-k-2}\big)+
\sum_{j=0}^{k} S_\st^\ww \big(0_j,1,0_{k-j},1,0_{d-k-2}\big)
 +\sum_{j=0}^{d-k-2} S_\st^\ww  \big(0_k,1,0_j,1,0_{d-k-j-2}\big) \\
&\,+\sum_{j=0}^{k-1} S_\st^\ww \big(0_j,(0)\ot(1),0_{k-j-1},1,0_{d-k-2}\big)
 +\sum_{j=0}^{d-k-3} S_\st^\ww  \big(0_k,1,0_j,(0)\ot(1),0_{d-k-j-3}\big),
\end{align*}
where we omit the signs if the corresponding argument pattern of the Euler sums is unique.
By Proposition~\ref{prop:Tsign}, all the terms on the shuffle side
contribute to the same sign in both MTVs. However, the last two terms are given by $j=d-2$ and $j=d-1$:
\begin{align*}
 &\, S_\sh^\ww  \big(2,1_{d-2},0;\mu_1,\dotsc,\mu_{d-2},  \nu\xi_{d-2} ,\nu\xi_{d-1}  \big)
+ S_\sh^\ww  \big(2,1_{d-2},1;\mu_1,\dotsc,\mu_{d-2}, \,  \mu_{d-1}\,  ,\nu\xi_{d-1}  \big).
\end{align*}
Thus the sign contribution of the two terms to MTVs are the same, but not for MSVs in general. Thus
by taking $k=0$ we see that sum formulas can be found for
\begin{equation}\label{equ:gF10..0;1}
\sum_{\substack{a,b\in\N,a\ge2\\ a+b=w-d}} 2^a T\big(a,b,1_d \big)
+\sum_{j=1}^{d} \sum_{\substack{\bfs\in\N^{j+2}, s_1\ge2 \\ |\bfs|=w-d+j}}  2^{a-1} T\big(\bfs,1_{d-j} \big).
\end{equation}
Similarly, assuming $d\ge 2$ and by taking $\bfy=(1,0;\nu)$ instead of $\bfy=(1;\nu)$ we may derive the weighted sum formula of the form
\begin{equation*}
d\sum_{\substack{a,b\in\N,a\ge2\\ a+b=w-d}} 2^a T\big(a,b,1_d \big)+
d\sum_{\substack{a,b,c\in\N,a\ge2\\ a+b+c=w-d+1}} 2^{a-1} T\big(a,b,c,1_{d-1} \big)
+\sum_{j=2}^{d} (d-j)\sum_{\substack{\bfs\in\N^{j+2}, s_1\ge2 \\ |\bfs|=w-d+j}}  2^{a-1} T\big(\bfs,1_{d-j} \big).
\end{equation*}
Thus multiplying \eqref{equ:gF10..0;1} by $d$ and subtracting it from  the above we arrive at the weighted sum formula of the form
\begin{equation}\label{equ:gF10..0;10}
 \sum_{j=2}^{d} j \sum_{\substack{\bfs\in\N^{j+2}, s_1\ge2 \\ |\bfs|=w-d+j}}  2^{a-1} T\big(\bfs,1_{d-j} \big)
\end{equation}
for all $d\ge 2$.

For examples, when $d=1$ and $d=2$ we can find the following formulas.
\begin{thm}\label{thm:Tab1_3W2a}
For all $w\ge 4$, set $u=w-2$ and $q=w-3$. Then we have
\begin{align*}
&\sum_{a+b=w-1,a\ge 2} 2^a \z(a,b,1)+\sum_{a+b+c=w,a\ge 2}  2^{a-1} \z(a,b,c)=w\z(w),\\
&\sum_{a+b=w-1,a\ge 2} 2^a T(a,b,1)+\sum_{a+b+c=w,a\ge 2}  2^{a-1} T(a,b,c)=2T(2)T(u),\\
&\sum_{a+b=w-2,a\ge 2} 2^a \z(a,b,1_2)+\sum_{a+b+c=w-1,a\ge 2}  2^{a-1} \z(a,b,c,1)\\
&\hskip5cm =v\z(u,1,1)-2\z(2,q,1)-\z(v,1)+3\z(2)\z(q,1)+\frac{v}2\z(u,2),\\
&\sum_{a+b=w-2,a\ge 2} 2^a T(a,b,1_2)+\sum_{a+b+c=w-1,a\ge 2}  2^{a-1} T(a,b,c,1)\\
&\hskip6cm
=2\Big(qM(u,\breve{1}_2)+qT(u,2)+T(3)T(q)-T(2,q,1)\Big),\\
& \sum_{a+b+c+d=w,a\ge 2}  2^{a} \z(a,b,c,d)=4\z(w) +2\z(3,q) - q\z(u,2),\\
& \sum_{a+b+c+d=w,a\ge 2}  2^{a} T(a,b,c,d) \\
&\hskip1cm = 2\Big( qT(u,1_2)-qM(u,\breve{1},\breve{1})-qT(u,2)+T(q,2,1)+T(q,1,2)+T(2,q,1)+T(3)T(q)\Big).
\end{align*}
\end{thm}

\begin{proof} Only the last identity needs to be explained. But by \eqref{equ:gF10..0;1} we have
\begin{align*}
&\sum_{a+b=w-2,a\ge 2} 2^a T(a,b,1_2)+\sum_{a+b+c=w-1,a\ge 2}  2^{a-1} T(a,b,c,1)+\sum_{a+b+c+d=w,a\ge 2}  2^{a-1} T(a,b,c,d) \\
&\hskip2cm =2qT(u,1_2)+2T(q,2,1)+2T(q,1,2)+2T(3)T(q).
\end{align*}
So we can find the last identity in the theorem by taking the difference of this and the fourth identity in the theorem.
\end{proof}

\begin{thm}\label{thm:M4Wc}
With notation as above, we have
\begin{equation*}
\sum_{a+b+c=w} \Big( 2^{b+c-1}(3^{a-1}-1)+2^{c} \Big)T(a,b,c)
+\sum_{a+b=w-1}\Big( 3^{a-1}(2^{b}+2)-2^{a}-2^{b}  \Big) T(a,b,1) =2^u\z(2)T(u).
\end{equation*}
\end{thm}
\begin{proof}Take $\gF(1,0;2)-\gF(1,2;0)$.
\end{proof}

\begin{thm}\label{thm:M4W2c}
With notation as above, we have
\begin{align*}
&\sum_{a+b+c+d=w,a\ge2} 2^c \, \z(a,b,c,d)= \frac{w-1}2 \z(w) + \frac{q}2\z(u,2)-\z(3,q)\\
&\sum_{a+b+c+d=w,a\ge2} 2^\tlc\,T(a,b,c,d)=\z(2)qT(u)-3S(3)T(q)-qS(u,2)-2S(q,3)-\frac43\z(\bar1)^3T(q)\\
&\hskip3cm  +2\z(\bar1)^2\Big(S(q,1)-qT(u)\Big)+2\z(\bar1)\Big(qS(u,1)-2M_3(\breveq,1,\bar1)\Big)\\
&\hskip3cm  +2\Big(M_2(\breveq,\bar1,\brt)+M_2(\breveq,2\ol{e_1},\bro)-qM_3(\breveu,1,\bar1)\Big)+4M_{34}(\breveq,1,\bar1,e_1e_2)
\end{align*}
\end{thm}

\begin{proof}
The identity $\gF(0,0,1;1)$ in Theorem~\ref{thm:ESgenFE} produces a weighted sum
formula that includes non-admissible terms. To remove these terms we may use the weighted sum
formula produced by $\gF(0,1;0,1)$.
\end{proof}

By combining Theorems~\ref{thm:Tab1_3W2a} and~\ref{thm:M4W2c} we immediately get the following result.
Note that the formula for MZVs was obtained by Machide as \cite[(1.11)]{Machide2015b}
\begin{cor}\label{cor:Machide(1.11)}
With notation as above, we have
\begin{align*}
&\sum_{a+b+c+d=w,a\ge2} (2^{a-1}+2^c) \z(a,b,c,d)= \frac{w+3}2 \z(w)
\end{align*}
\end{cor}

\begin{thm}\label{thm:M2121_2101}
With notation as above, we have
\begin{align*}
\sum{}' & 2^{a+b-1} \z(a,b,c,d)+\sum{}' 2^{a} \z(a,b,1,d)=w\z(v,1)+\z(3,q)-\frac{q}2\z(u,2)+w\z(2)\z(u)-\frac{w-1}2\z(w),\\
\sum{}' & 2^{a+b-1} T(a,b,c,d)+\sum{}' 2^{a} T(a,b,1,d)=2q\z(2)T(u) + \z(3)T(q) - 2qS(u, 2) - 4S(q, 3),\\
\sum{}' & 2^{a+b-1} \z(\bara,b,c,d)+\sum{}' 2^{a} \z(\bara,b,1,d)=
T(u,2)-u\z(\bar1)T(v)-\Big(2\z(\bar1)^2-\frac12\z(2)\Big)T(u)+\z(3,q)+\frac{q}2\z(2,u)\\
 -& 4M_{23}(\breveu,e_1,\bar1)+4\z(\bar1)M_2(\breveu,e_1)+M_1(u,\breve2)+2u\z(v,1)
    -u\z(\barv,\bar1)+2\z(\barv,1)-\z(\barw)-\Big(u+\frac12\Big)\z(u)\z(2).
\end{align*}
\end{thm}
\begin{proof}
These follow from $\gF(1,0,1;1)-\gF(0,1;0,1)/2$.
\end{proof}

\begin{thm}\label{thm:M4W2221_2211}
With notation as above, we have
\begin{align*}
\sum{}' (2^{a-1}-1)2^b(2^c+1) & \z(a,b,c,d)=2u\z(w)+q\z(u,2)-2\z(3,q),\\
\sum{}' (2^{a-1}-1)2^b(2^c+1) & T(a,b,c,d)=
4q\Big(T(u,2)+\z(2)T(u)\Big)-4\z(2)S(q,1)+\Big(8\z(2)\z(\bar1)-2\z(3)\Big)T(q)\\
+&8\Big(T(q,2,1)+\z(\bar1)T(q,2)+T(q,3)+qM_2(\breveu,e_1,\breve1)+M_2(\breveq,e_1,\breve2)-M_3(q,\breve2,\bar1)\Big)\\
+& 16\Big(\z(\bar1) M_2(\breveq,e_1,\breve1)+M_{23}(\breveq,e_1,\ol{e_4},\breve1) \Big).
\end{align*}
\end{thm}
\begin{proof}
These follow from $\gF(1,1;1,1)-2\gF(0,1,1;1)$.
\end{proof}

Combining Theorems~\ref{thm:Tab1_3W2a} and \ref{thm:M4W2221_2211} we immediately obtain the following
weighted sum formula proved for general depth by Guo and Xie \cite{GuoXi2009}.
This particular result also appeared as \cite[(1.10)]{Machide2015b}.
\begin{cor}\label{cor:z4W222}
With notation as above, we have
\begin{align*}
\sum{}' (2^{a+b+c-1}+2^{a+b-1}+2^a-2^{b+c}-2^c) \z(a,b,c,d)=w\z(w).
\end{align*}
\end{cor}

The second formula in the next theorem can be regarded as a depth four analog of \cite[Corollary 6.2]{BCJXXZ2020}, originally a conjecture of Kanecko and Tsumura \cite[Conjecture~4.6]{KanekoTs2019}. The first formula for MZVs appeared as \cite[(1.12)]{Machide2015b}.

\begin{thm}\label{thm:M4W432_2321}
With notation as above, we have
\begin{align*}
\sum{}'& 2^{a+c}(2^{a-1}3^b-3^b-1) \z(a,b,c,d)=\left( 2+\frac{(w-3)(w+2)(w+7)}{12} \right)\z(w) ,\\
\sum{}'& 2^{a+c}(2^{a-1}3^b-3^b-1) T(a,b,c,d)=4\binom{v}{3} T(w),\\
\sum{}'& 2^{a+c}(2^{a-1}3^b-3^b-1) \z(\bara,b,c,d)
=\frac{(w-1)(w^2-5w+18)}{12}\z(w)+w(w-3)\z(\barw)+4\big(T(v,1)+T(u,2)\big) \\
&-\big(4\z(\bar1)^2+2\z(2)\big)T(u)+4q\big(M_2(\brevev,e_1)-\z(\bar1)T(v)\big)+8\big(\z(\bar1)M_2(\breveu,e_1)-M_{23}(\breveu,e_1,\bar1)\big).
\end{align*}
\end{thm}
\begin{proof}
These follow from $\frac12\gF(1;1;1;1)-2\gF(0,1;1;1)+2\gF(1,0,1;1)+\gF(0,1;0,1)$.
\end{proof}

The next result can also be regarded as a depth four analog of \cite[Corollary 6.2]{BCJXXZ2020} although it is not as clean.

\begin{thm}\label{thm:M4W432_3321_1221}
With notation as above, we have
\begin{align*}
& \sum{}' \big[(4^{a-1}-3^{a-1})3^b2^c-(3^{a-1}-1)2^b(2^c+1)\big] \z(a,b,c,d)=
\left(\frac14\binom{v}{3}-2v+1\right)\z(w)+\z(3,q)-\frac{q}2 \z(u,2),\\
&\sum{}' \big[(4^{a-1}-3^{a-1})3^b2^c-(3^{a-1}-1)2^b(2^c+1)\big] T(a,b,c,d)
=2\binom{v}{3}T(w)+\left(8\z(\bar1)^3-\frac12\z(3)\right)T(q)\\
+& 2uq\Big(2\z(\bar1)T(v)-S(v,1)\Big)
+2q\Big(S(u,2)+ (6\z(\bar1)^2-\z(2))T(u)\Big)
+24\Big(\z(\bar1)M_3(\breveq,1,-1)-M_{34}(\breveq,1,\bar1,e_1e_2)\Big)\\
+& 8M_2(\breveq,3e_1)+12\Big(qM_3(\breveu,1,\bar1)-q\z(\bar1)S(u,1)
-\z(\bar1)^2S(q,1)-M_2(\breveq,\ol{e_1},\breve2)-M_2(\breveq,2\ol{e_1},\breve1)\Big).
\end{align*}
\end{thm}
\begin{proof}
These follow from $\frac14\gF(1;1;1;1)-\gF(1,1;1;1)+2\gF(0,1,1;1)$.
\end{proof}

The last few results in this section are a little different in flavor. In these weighted sum formulas, each weight
is given by a combination of product of a polynomial and a 2-power.

\begin{thm}\label{thm:T3Wa2a2b_3Wa2a_3Wb2a}
For all $w\ge 4$, set $v=w-1$, $u=w-2$, $q=w-3$, and $\sum{}'=\sum_{a+b+c=w,a\ge2}$. Then we have
\begin{align*} 
&\sum{}' \big( \tla\cdot 2^{a+b}+(\tla+2\tlb)2^a\big)T(a,b,c)=8\Big(qT(2)T(u)+\z(\bar1)T(2)T(q)+T(q,3)\\
&\hskip3cm -M_2(\barq,\ol{e_1},\bar2)-M_2(\barq,2\ol{e_1},\bar1)-T(2)M_2(\barq,\ol{e_1})\Big)-4T(3)T(q).
\end{align*}
\end{thm}
\begin{proof}
One may apply $\gF_1(1,1;1)$ in Theorem~\ref{thm:ESgenFE} first and then use Theorems~\ref{thm:EulerSum1bcWb_c}
to remove the non-admissible terms $\tlb\, T(1,b,c)$.
\end{proof}

\begin{thm}\label{thm:T3Wb2a2b_3Wc2a2b_3Wc2a}
With notation as above, we have
\begin{align*} 
&\sum{}' \big( (\tlb+\tlc)(2^\tla-1)2^\tlb+\tlc\cdot 2^\tla \big)T(a,b,c)\\
&\hskip1cm=q\Big( 2\z(\bar1)S(u,1)+uS(v,1)-2u\z(\bar1)T(v)-(2\z(\bar1)^2+\z(2))T(u)-2M_3(\breveu,1,\bar1)\Big).
\end{align*}
\end{thm}
\begin{proof}
Use $\gF_2(1,1;1)-\gF_2(0,1;1)$ in Theorem~\ref{thm:ESgenFE} first, then use Theorems~\ref{thm:EulerSum1bcWb_c}
to remove the non-admissible terms $\tlc\, T(1,b,c)$ .
\end{proof}

\begin{thm}\label{thm:T3Ww2a2b_3Wa2a}
With notation as above, we have
\begin{align*} 
&\sum{}' \big( (q\cdot 2^\tla-q+\tla)2^\tlb+\tla\cdot 2^{a-2}  \big)T(a,b,c)=
\z(2)\Big(3M_2(\breveq,\ol{e_1})-qT(u)-3T(q)\Big)\\
& \hskip1cm +uq\Big(S(v,1)-2\z(\bar1)T(v)\Big)+2q\Big(\z(\bar1)S(u,1)-\z(\bar1)^2T(u)-M_3(\breveu,1,\bar1)\Big)\\
& \hskip1cm +2\Big(M_2(\breveq,2\ol{e_1},\breve1)+M_2(\breveq,\ol{e_1},\breve2)-T(q,3)\Big)+T(3)T(q).
\end{align*}
\end{thm}
\begin{proof}
Use $\gF_3(1,1;1)-\gF_2(0,1;1)$ in Theorem~\ref{thm:ESgenFE}.
\end{proof}

\begin{cor}\label{cor:T3Wb2a_3Wc2a}
With notation as above, we have
\begin{align*} 
\sum{}' (\tlb+\tlc) 2^\tla\cdot T(a,b,c)=&\,
3\z(2)\Big(2\z(\bar1)T(q)+qT(u)-2M_2(\breveq,\ol{e_1})\Big)\\
&\, +4\Big(T(q,3)-M_2(\breveq,\ol{e_1},\breve2)-M_2(\breveq,2\ol{e_1},\breve1)\Big)-2T(3)T(q).
\end{align*}
\end{cor}
\begin{proof}
This follows from Theorems~\ref{thm:T3Wa2a2b_3Wa2a_3Wb2a}-\ref{thm:T3Ww2a2b_3Wa2a}.
\end{proof}

\begin{thm}\label{thm:M3Wa2b}
With notation as above, we have
\begin{align*} 
&\sum{} \tla\cdot 2^b \z(a,b,c)=q\z(u,2)-2u\z(v,1)+v\z(2,u),\\
&\sum{} \tla\cdot 2^b \z(\bara,\barb,c)=2q\Big(\z(\baru,\bar1,\bar1)-\z(\baru,\bar1,1)\Big)+q\Big(\z(\bar2,\baru)+\z(\baru,2)\Big)\\
&\hskip1cm   -2\Big(u\z(v,1)+\z(\bar2,\barq,\bar1)+\z(2,\bar1,\barq)+\z(\barq,\bar2,1)+\z(v,\bar1)\Big)+S(3)\z(\barq),\\
&\sum{} \tla\cdot 2^b \z(\bara,b,c)=v\Big(\z(\bar2)\z(\baru)-\z(w)-2\z(\barv,1)\Big)+2\z(\bar1)\Big(\z(\barq,2)-T(v)-\z(\barq,\bar2)\Big)+T(3)\z(\barq)\\
&\hskip1cm    +2\Big(\z(\barq,\bar2,\bar1)-\z(\barq,2,1)-\z(\barq,\bar3)+\z(\barq,3)+\z(v,\bar1)\Big)-q\z(\baru,\bar2)+(q-2)\z(\baru,2),\\
&\sum{} 2^b \z(a,\barb,c)=2q\Big(\z(u,\bar1,\bar1)-\z(u,\bar1,1)-\z(\bar2)\z(u)\Big)
    +\z(2)\Big(3\z(\bar1,q)+2\z(\baru)\Big)   \\
&\hskip1cm -v\Big(2\z(\barv,1)+\z(w)\Big) +2\z(\bar1)\Big(\z(\barv)+\z(q,\bar2)-\z(q,2)-\z(v)\Big) \\
&\hskip1cm    +2\Big(\z(q,\bar2,\bar1)-\z(\baru,2)+\z(q,\bar1,\bar2)-\z(q,\bar1,2)+\z(q,2,\bar1)+\z(v,\bar1)\Big)-4\z(q,\bar2,1),\\
&\sum{} \tla\cdot 2^{\tlb} T(a,b,c)=q\z(2)T(u).
\end{align*}
\end{thm}
\begin{proof}
Use $\gF_1(0,1;1)$ in Theorem~\ref{thm:ESgenFE}, Theorems~\ref{thm:MMVa2c} and~\ref{thm:Ta1cWac}.
\end{proof}

We end this section by remarking that sometimes we may obtain weighted sum formulas for Euler sums
but not MTVs or other variants because we cannot remove the non-admissible terms.
For example,  in \cite[Thm. 4.10]{BCJXXZ2020} we find the sum formula for
$$\sum{}   2^b \z_*(a,b,c),\quad
\sum{} 2^b \z(\bara,\barb,c),\quad
\sum{} 2^b \z(\bara,b,c),\quad
\sum{} 2^b \z_*(a,\barb,c),
$$
but not the other four alternating sign types since they are pairwise intertwined, namely, the sum $\z(\bara,\barb,\barc)+\z(\bara,b,\barc)$
cannot be split, neither can  $\z(a,\barb,\barc)+\z(a,b,\barc)$. The non-admissible terms are provided by
restricted weighted sums such as $\sum_{b+c=w} 2^b \z_*(1,b,c)$ for which we cannot find a closed formula.

\section{Weighted and restricted sum formulas: general depth}
In \cite{GuoXi2009} Guo and Xie obtained the general weighted sum formula for multiple zeta values:
for positive integers $d,s_1,\ldots,s_{d}$, let
\begin{align*}
\mathcal{C}(s_1,\ldots,s_{d})=&\sum\limits_{j=1}^{d}2^{s_1+\cdots+s_{j}-j}+2^{s_1+\cdots+s_{d}-d}
=\sum\limits_{j=1}^{d-1}2^{s_1+\cdots+s_j-j}+2^{s_1+\cdots+s_{d}-d+1}.
\end{align*}
Then
\begin{equation}\label{Eq:WSum-GuoXie-MZV}
\sum\limits_{s_1+\cdots+s_d=w\atop s_i\ge 1,s_1\ge 2}[\mathcal{C}(s_1,\ldots,s_{d-1})-\mathcal{C}(s_2,\ldots,s_{d-1})]\zeta(s_1,\ldots,s_d)=w\zeta(w).
\end{equation}

We can now generalize \eqref{Eq:WSum-GuoXie-MZV} to Euler sums with any fixed pattern of alternating signs.
But first we need to find a few preliminary restricted weighted sum formulas.

\begin{thm} \label{thm:zab11d}
For all $u\ge 3$, let $q=u-1$.  Set  $\sum{}'=\sum_{a+b=u,a,b\in\N,a\ge2}$. Then for all $d\ge 0$ we have
\begin{align*}
\sum{}'\z(a,\bar{b},\bar1,1_d)=&\, \z(\barq,1,\bar1,1_d)-\z(\barq,\bar1,1,1_d)+\z(\baru,\bar1,1_d)+\z(\barq,\bar2,1_d)
+\sum_{j=1}^d \z(\barq,\bar1,1_{j-1},2,1_{d-j}),\\
\sum{}'\z(a,b,\bar1,\bar1,1_d)=&\, \z(q,\bar1,1,\bar1,1_d)-\z(q,\bar1_2,1,1_d)+\z(u,\bar1_2,1_d)+\z(q,\bar2,\bar1,1_d)+\z(q,\bar1,\bar2,1_d) \\
&\, +\sum_{j=1}^d \z(\barq,\bar1_2,1_{j-1},2,1_{d-j}).
\end{align*}
\end{thm}

For arbitrary depth, we have the following result.
\begin{thm} \label{thm:zab1..1Wa}
For all $u\ge 3$, let $q=u-1$ and $p=u-2$. Then for all $d\ge 0$ we have
\begin{equation*}
\sum_{a+b=u,\, a,b\in\N} \tla\, \z_\sh(a,b,1_d;\bfmu)=I_{\ref{thm:zab1..1Wa}}+\II_{\ref{thm:zab1..1Wa}}-\III_{\ref{thm:zab1..1Wa}}-IV_{\ref{thm:zab1..1Wa}},
\end{equation*}
where by setting $\nu_2=\mu_1\mu_2$ and $\nu_j=\mu_j$ for all $j\ne 2$
\begin{align*}
I_{\ref{thm:zab1..1Wa}}=&\, \z(2,p,1_d;\nu_1,\nu_2,\dots,\nu_{d+2})+\sum_{j=0}^d \z(p,1_j,2,1_{d-j};\nu_2,\dots,\nu_{j+2},\nu_1,\nu_{j+3},\dots,\nu_{d+2}),\\
\II_{\ref{thm:zab1..1Wa}}=&\, \z(u,1_d;\nu_1\nu_2,\nu_3,\dots,\nu_{d+2})+\sum_{j=0}^{d-1} \z(p,1_j,3,1_{d-1-j};\nu_2,\dots,\nu_{j+2},\nu_1\nu_{j+3},\nu_{j+4},\dots,\nu_{d+2}),\\
\III_{\ref{thm:zab1..1Wa}}=&\,\sum_{j=0}^d  \sum_{k=0}^j \z(p,1_k,2,1_{d-k};\nu_2,\dots,\nu_{j+2},\nu_1\dots\nu_{j+2},\nu_1\dots\nu_{j+3}, \nu_{j+4},\dots,\nu_{d+2}),\\
IV_{\ref{thm:zab1..1Wa}}=&\,\sum_{j=0}^d p\z(q,1_{d+1};\nu_2,\dots,\nu_{j+2},\nu_1\dots\nu_{j+2},\nu_1\dots\nu_{j+3}, \nu_{j+4},\dots,\nu_{d+2}).
\end{align*}
In particular, for MZVs we have
\begin{align*}
\sum{}'\tla\,\z(a,b,1_d)= &\, \z(u,1_d)+\z(2,p,1_d)-p(d+1)\z(q,1_{d+1}) \\
&\,  +\sum_{j=0}^{d-1} \Big(\z(p,1_{j},3,1_{d-1-j})-(d-j)\z(p,1_{j},2,1_{d-j}) \Big).
\end{align*}
\end{thm}
\begin{proof} Take $\bfx=(1,0_d;\nu_2,\dots,\nu_{d+2})$ and $\bfy=(0;\nu_1)$ in Theorem~\ref{thm:ESgenFE}.
First observe that the modifying terms have no contribution
to in $\gF_{d+2}(1,0_d;0)$ since the weight $w\ge d+3$. Next, we see that the shuffle side is equal to
\begin{align*}
\sum_{a+b=u,\, a,b\in\N}  \tla\, \z_\sh(a,b,1_d;\bfmu)+\III+IV
\end{align*}
and the stuffle side is given by $I+\II$ where the terms in $\II$ are produced by ``stuffing''. The theorem follows immediately.
\end{proof}

\begin{thm} \label{thm:zab1..1Wb}
For all $u\ge 3$, let $q=u-1$ and $p=u-2$. Then for all $d\ge 0$ we have
\begin{equation*}
\sum_{a+b=u,\, a,b\in\N} \tlb\, \z_\sh(a,b,1_d;\bfmu)=p(I_{\ref{thm:zab1_d}}+\II_{\ref{thm:zab1_d}})-
I_{\ref{thm:zab1..1Wa}}-\II_{\ref{thm:zab1..1Wa}}+\III_{\ref{thm:zab1..1Wa}},
\end{equation*}
where $I_{\ref{thm:zab1_d}}$ and $\II_{\ref{thm:zab1_d}}$ are given in Theorem~\ref{thm:zab1_d}, and
$I_{\ref{thm:zab1..1Wa}}$, $\II_{\ref{thm:zab1..1Wa}}$ and $\III_{\ref{thm:zab1..1Wa}}$ are given in Theorem~\ref{thm:zab1..1Wa}.
In particular, for MTVs we have
\begin{align*}
\sum{}'\tlb\,\z(a,b,1_d)=&\,  (p-1)\z(u,1_d) -\z(2,p,1_d)+p(d+1)\z(q,1_{d+1}) \\
&\, +\sum_{j=0}^{d-1} \Big( p \z(q,1_j,2,1_{d-1-j})
 - \z(p,1_{j},3,1_{d-1-j})+(d-j)\z(p,1_{j},2,1_{d-j}) \Big).
\end{align*}
\end{thm}
\begin{proof} Note that
\begin{equation*}
\sum_{a+b=u,\, a,b\in\N} \tlb\, \z_\sh(a,b,1_d;\bfmu)=p\sum_{a+b=u,\, a,b\in\N} \z_\sh(a,b,1_d;\bfmu)-\sum_{a+b=u,\, a,b\in\N} \tla\, \z_\sh(a,b,1_d;\bfmu).
\end{equation*}
The theorem follows from Theorems~\ref{thm:zab1_d} and \ref{thm:zab1..1Wa} quickly since $p \cdot \III_{\ref{thm:zab1_d}}=IV_{\ref{thm:zab1..1Wa}}$.
\end{proof}

\section{Concluding remarks}
We remark that it is possible to generalize the various restricted/weighted sum formulas to colored MZVs, i.e., special values of multiple polylogarithms at $N$-th roots of unity. We plan to work on this project in the near future.

In \cite{KanekoTa2013} Kaneko and Tasaka considered the double zeta values and double Eisenstein series at level 2. Yuan and and the author \cite{YuanZh2014b} generalized these to level $N$ by considering the subseries of a MZV series where the summation indices run through any fixed congruence class modulo $N$.  Similarly, the MMVs are clearly the multiple variable version at level 2 and therefore can be extended to level $N$. In \cite{YuanZh2014b} we defined their two ways of regularization and the corresponding double shuffle relations so that it is conceivable that many results in this paper may be extended to arbitrary levels using the results on colored MZVs.

We now propose a few questions with some partial answers. 

\medskip
\noindent
\textbf{Question 1.}
For all $d\ge 5$ and $k\le d-2$, are there restricted sum formulas of Euler sums
of depth $d$ when $k$ components are restricted to 1's? 

\smallskip
\emph{Partial Answers.} Yes if $k=1,2$ or if all $k$ components are consecutive; but unknown in other cases for general $d\ge 5$.  
For examples, we know the formulas for the following restricted sums exist:
\begin{align*}
\sum{}' \z(a,\barb,1,c,d),   \qquad
\sum{}' \z(a,b,1,\barc,1), \qquad
\sum{}' \z(a,\barb, 1_\ell,\barc),
\end{align*} 
where $\sum'=\sum_{a+b+c+d=w-1,a\ge 2}$; but not for 
$$ \sum_{a+b=w,a\ge 2} \z(a,1,1,b,1).$$ 
 
\medskip
\noindent
\textbf{Question 2. }
For $r\ne 0,1,2$, are there formulas for the weighted sums of the following form
$$\sum\limits_{\bfs\in\N^d,\ |\bfs|=w,\ s_1\ge 2} r^{s_1} \z(\bfs)?$$ 

\smallskip
\emph{Partial Answers.}  Yes in two cases: (i) $d=r=2$ or (ii) $d=2$, $r=-1$, and $w$ is even; but unknown for all other cases.

\medskip
\noindent
\textbf{Question 3.}
Is there a restricted sum formula for
$$f(w):=\sum\limits_{a+b=w,\  a,b \text{ even}}  \z(a,\barb)?$$ 

\bigskip

\noindent
\textbf{Acknowledgement.} The author would like to thank the support by the Jacobs Prize from the Bishop's School while the work was done.

\bigskip

\begin{center}
\textbf{\large Appendix}
\end{center}

\medskip
We will now summarize the main results of this paper and some previous related formulas in three tables. Recall that
$$ \gF\big((\bfx;\bfmu);(\bfy;\bfnu)\big):=F_\sh(\bfx;\bfmu)F_\sh(\bfy;\bfnu)-\rho(F_*(\bfx;\bfmu))\rho(F_*(\bfy;\bfnu))=0.
$$
In the tables below we will use $\gF$ to denote the weight $w$ part of this identity while suppressing the alternating signs.
The subscripts of $\gF_j$ means to take the partial derivative with respect to the $j$th variable.
The notation $\sum{}' f(\bfs)$  implies that $|\bfs|=w$ and
the first argument in $\bfs$ must be at least 2 while $\sum{}_\bfs$
does not have restriction on the first argument.
Formulas of type $\sum{}' A(\bfs)$ means Euler sums of all alternating types can be produced, which implies that similar
formulas for all MMVs can also be derived; $\sum{} E(\bfs)$ means Euler sums of all alternating types can be derived
but not MMVs since the first argument is allowed to be 1 (and thus the regularized values appearing in the formulas cannot be removed);
$\sum M(\bfs)$ indicates that formulas for MTVs, MZVs and Euler sums of the type
$\z(\bfs;-1,1,\dots,1)$ (and maybe a few other types of alternating signs) can be derived;  $\sum \z(\bfs)$
means only MZVs are involved.  Moreover, for any positive integer $n$ we set $\tilde{n}=n-1$. And, we leave a few formulas as simple exercises.

\begin{table}[!h]
{
\begin{center}
\begin{tabular}{  |c|c|c| } \hline
      Method     & Weighted/Restricted Sum Formula  & Reference  \\ \hline
$\scriptstyle \gF(0;1)$  & $\sum{}'  A(a,b)$  &   \cite[Thm.~4.2, Cor.~4.3]{BCJXXZ2020}  \\ \hline
$\scriptstyle \gF(-1;1)$ & $\sum{}'_{2\nmid ab} \z(a,b)$, $\sum{}'_{2\nmid ab} \z(\bara,\barb)$, $\sum{}'_{2\nmid ab} t(\bara,\barb)$ &    \cite[Thm.~1]{GanglKaZa2006}\\ \hline
$\scriptstyle \gF(-1;1)$ & $\sum{}'_{2|a, 2|b} \z(a,b)$,  $\sum{}'_{2|a, 2|b} \z(\bara,\barb)$, $\sum{}'_{2|a, 2|b} t(\bara,\barb)$  &    \cite[Thm.~1]{GanglKaZa2006}\\ \hline
$\scriptstyle \gF_1(0;1),\gF_2(0;1)$  & $\sum{}' \tla A(a,b)$, $\sum{}' \tlb A(a,b)$  &  Simple Exercise  \\ \hline
$\scriptstyle \gF_{i,j}(1;0),\ i,j=1,2$ & $\sum{}' \tla^2 A(a,b)$, $\sum{}' \tla\tlb A(a,b)$, $\sum{}' \tlb^2 A(a,b)$  &   Thm.~\ref{thm:MMVabWallQ} \\ \hline
$\scriptstyle \gF(1;1)$  & $\sum{}' 2^\tla M(a,b)$   &  \cite[Thm.~4.4]{BCJXXZ2020}, \cite[Thm. 3.2]{KanekoTs2019} \\ \hline
$\scriptstyle \gF(2;1)$  & $\sum{}'  (3^\tla-1)(2^\tlb+1) \z(a,b)$  &  \cite[Thm.~3.7]{BCJXXZ2020} \\ \hline
\end{tabular}
\end{center}
}
\caption{Weighted/restricted sum formulas in depth two.}
\label{Table: Weighted/restricted sum formulas in depth two.}
\end{table}

\begin{table}[!h]
{
\begin{center}
\begin{tabular}{  |c|c|c| } \hline
      Method     & Weighted/Restricted Sum Formula  & Reference  \\ \hline
$\scriptstyle \gF(1,0;0)$  & $\sum{}' A(a,b,1)$  &  \cite[Thm.~5.2, Cor.~5.3]{BCJXXZ2020}\\ \hline
$\scriptstyle \gF(0,1;0)$  & $\sum{} E(1,b,c)$  &  \cite[Thm.~5.4]{BCJXXZ2020}  \\ \hline
$\scriptstyle \gF(0,0;1)-\gF(1,0;0)$  & $\sum{}' A(a,b,c)$  &  \cite[Thm.~5.5]{BCJXXZ2020} \\ \hline
$\scriptstyle \gF(0,0;1)$  & $\sum{}' T(a,b,c)+\sum{}' T(a,b,1)$  &  \cite[Cor.~5.6]{BCJXXZ2020}, \cite[Thm.~3.3]{KanekoTs2019}, \\ \hline
$\scriptstyle \gF(1,1;0)$  & $\sum{}' A(a,1,c)$  &  \cite[Thm.~5.7]{BCJXXZ2020} \\ \hline
$\scriptstyle \gF(0,1;1)$  & $\sum{}' 2^b M(a,b,c)$  &  \cite[Thm.~5.8]{BCJXXZ2020} \\ \hline
$\scriptstyle \gF(1,1;1)$  & $\sum{}' (2^{a}+2^{a+b}) M(a,b,c)$  &  Simple Exercise \\ \hline
$\scriptstyle \gF(1,0;1)$  & $\sum{}' 2^{\tla} M(a,b,c)+\sum{}' 2^{a} M(a,b,1)$  & Simple Exercise \\ \hline
$\scriptstyle \gF(1;1;1)/3-\gF(0,1;1)$  & $\sum{}' 2^b (3^{\tla}-1) M(a,b,c)$  &  \cite[Thm.~6.1, Cor.~6.2]{BCJXXZ2020}, \cite[Conj.~4.6]{KanekoTs2019} \\ \hline
$\scriptstyle \gF_2(1,0;0), \gF_1(0,1;0)$  & $\sum_{}' A(a,b,2)$, $\sum_{}' A(2,b,c)$  & Thm.~\ref{thm:MMVab2_T2bc}\\ \hline
$\scriptstyle E(2)\sum_{a+b=w-2}E(a,b)$  & $\sum_{}' A(a,2,b)$  & Thm.~\ref{thm:MMVa2c}\\ \hline
$\scriptstyle \gF_3(1,0;0), \gF_3(1,0;0)$  & $\sum{} \tla A(a,b,1)$, $\sum{}' \tlb A(a,b,1)$ & Thm.~\ref{thm:MMVab1Wa_ab1Wab} \\ \hline
$\scriptstyle \gF_2(1,1;0)$  & $\sum{} \tla A(a,1,c)$, $\sum{}' \tlc A(a,1,c)$ & Thm.~\ref{thm:Ta1cWac}\\ \hline
$\scriptstyle \gF_3(0,1;0), \gF_2(0,1;0)$ & $\sum{} \tlb E(1,b,c)$, $\sum{}\tlc E(1,b,c)$ & Thm.~\ref{thm:EulerSum1bcWb_c}\\ \hline
$\scriptstyle \gF_i(0,0;1), i=1,2,3$& $\sum{} \tla A(a,b,c)$, $\sum{}' \tlb A(a,b,c)$, $\sum{}'\tlc A(a,b,c)$ & Thm.~\ref{thm:MMV3Wa_3Wb_3Wc}\\ \hline
$\scriptstyle \gF(0,1;-1)+\gF(1,0;-1)$& $\sum{}'_{2\nmid ac} \z(a,1,c)$, $\sum{}'_{2|a,2|c} \z(a,1,c)$ & Cor.~\ref{cor:a1cEvenOdd}\\ \hline
$\scriptstyle \gF(-1,1;0)$ & $\sum{}'[(-1)^c-(-1)^a] \z(a,b,c)$  & Thm.~\ref{thm:zabcW_11_1}\\ \hline
$\scriptstyle\gF(1,0;2)-\gF(1,2;0)$& $\aligned &\sum{}'\Big( 2^{b+c-1}(3^{a-1}-1)+2^{c} \Big)T(a,b,c)\\
            +&\sum{}'\Big( 3^{a-1}(2^{b}+2)-2^{a}-2^{b}  \Big) T(a,b,1)\endaligned $  & Thm.~\ref{thm:M4Wc}\\ \hline
$\scriptstyle \gF_1(1,1;1)$& $ \sum{}'\big( \tla\cdot 2^{a+b}+(\tla+2\tlb)2^a\big) M(a,b,c)$  &  Thm.~\ref{thm:T3Wa2a2b_3Wa2a_3Wb2a}\\ \hline
$\scriptstyle\gF_2(1,1;1)-\gF_2(0,1;1)$ & $\sum{}' \big( (\tlb+\tlc)(2^\tla-1)2^\tlb+\tlc\cdot 2^\tla \big)M(a,b,c)$  &  Thm.~\ref{thm:T3Wb2a2b_3Wc2a2b_3Wc2a}\\ \hline
$\scriptstyle\gF_3(1,1;1)-\gF_2(0,1;1)$ & $\sum{}'\big( (q\cdot 2^\tla-q+\tla)2^b+\tla\cdot 2^\tla\big) M(a,b,c)$ &  Thm.~\ref{thm:T3Ww2a2b_3Wa2a}\\ \hline
  & $\sum{}' (\tlb+\tlc) 2^\tla\cdot M(a,b,c)$ & Cor.~\ref{cor:T3Wb2a_3Wc2a}\\ \hline
$\scriptstyle \gF_1(0,1;1)$ & $\sum{}' \tla\cdot 2^b M(a,b,c)$ &  Thm.~\ref{thm:M3Wa2b}\\  \hline
\end{tabular}
\end{center}
}
\caption{Weighted/restricted sum formulas in depth three.}
\label{Table: Weighted/restricted sum formulas in depth three.}
\end{table}

\begin{table}[!h]
{
\begin{center}
\begin{tabular}{  |c|c|c| } \hline
      Method     & Weighted/Restricted Sum Formula  & Reference  \\ \hline
$\scriptstyle \gF(0_k,1,0_l;0)$  & $\sum{} E(1_k,a,b,1_l) $  &  \eqref{equ:zMiddle2termSum} \\ \hline
$\scriptstyle \gF(1,0_d;0)$  & $\sum{}' A(a,b,1_d) $  &  \eqref{equ:M2termSum} \\ \hline
$\scriptstyle \gF(1,0_d;0)$  & $\sum{} E(a,b,1_d) $  &  Thm.~\ref{thm:zab1_d} \\ \hline
$\scriptstyle \gF(1,0,0;0)$  & $\sum{}' A(a,b,1,1) $  &  Cor.~\ref{cor:zab11}, Cor.~\ref{cor:zbarab11}, Thm.~\ref{thm:Tab11} \\ \hline
$\scriptstyle \gF(0_k,1,0_l;0,0)$  & $\sum{} E(1_k,a,b,c,1_l) $  &  \eqref{equ:3termSum} \\ \hline
$\scriptstyle \gF\big(0_{k},1,0_l;0_{d-1} \big)$ & $\sum{}_{|\bfs|=d} E\big(1_{k},\bfs,1_l\big)$  &  \eqref{equ:ElltermSum},\cite{EieLiOng2009,Granville1997b} \\ \hline
$\scriptstyle \gF\big(0_{k},1,1,0_l;0\big)$ & $\sum E\big(1_{k},a,1,b,1_{l}\big)$  &  \eqref{equ:z1_ka1b1_l} \\ \hline
$\scriptstyle \gF(1,1,0;0)$ & $\sum{}' A(a,1,b,1)$  &  \eqref{equ:z1_ka1b1_l} \\ \hline
$\scriptstyle \gF\big(0_{k},1,1,0_l;0_d\big)$  & $\sum_{} E(1_{k},a,1_d,b,1_{l})$  &  \eqref{equ:z1_ka1_db1_l}\\ \hline
$\scriptstyle \gF\big(1,1;0,0\big)$  & $\sum_{}' A(a,1,1,b)$  & Thm.~\ref{thm:Ma11d}\\ \hline
$\scriptstyle \gF\big(1,1,0;0\big)$  & $\sum_{}' A(a,b,c,1)$  &  Thm.~\ref{thm:M31}\\ \hline
$\scriptstyle \gF(1,0;0,1)-\gF(0,1,0;1)$  & $\sum_{}' A(a,1,c,d)$  & Thm.~\ref{thm:Ma1cd}\\ \hline
$\scriptstyle \gF(1,0,1;0)$   & $\sum_{}' A(a,b,1,d)$  & Thm.~\ref{thm:Mab1d}\\ \hline
$\scriptstyle \gF(0,0,0;1)$   & $\sum_{}' A(a,b,c,d)$  & Thm.~\ref{thm:M4}\\ \hline
$\scriptstyle \gF_i(1_i;-1), i=1,2,3$& $\sum_{\dep(\bfs)=d}\sum_{i=1}^d (-1)^{s_i} \zeta_\st(\bfs)$, $2\le d\le 4$ & Thm.~\ref{thm:zdW_1}, \cite[Thm.~3]{ShenCai2012}\\ \hline
$\scriptstyle \gF(1,0_d;1)$ & $\sum_{j=2}^{d+2} \sum'_{\dep(\bfs)=j}  2^{\tla+[2/j]} M\big(\bfs,1_{d+2-j} \big)$  &  \eqref{equ:gF10..0;1} \\ \hline
$\scriptstyle \gF(1,0_d;1)+\gF(1,0_{d-1};1,0)$ & $\sum_{j=2}^{d} j \sum'_{\dep(\bfs)=j+2}2^\tla M\big(\bfs,1_{d-j} \big)$  &  \eqref{equ:gF10..0;10}, Thm.~\ref{thm:Tab1_3W2a} \\ \hline
$\scriptstyle\gF(0,0,1;1)$ & $\sum{}'2^c \, M(a,b,c,d)$  &  Thm.~\ref{thm:M4W2c}\\ \hline
$\scriptstyle\gF(0,0,1;1)$ & $\sum{}'(2^{a-1}+2^c) M(a,b,c,d)$  & Cor.~\ref{cor:Machide(1.11)}, \cite[Thm.~3.3.2]{Eie2013}, \cite[(1.11)]{Machide2015b}\\ \hline
$\scriptstyle\gF(1,0,1;1)-\gF(0,1;0,1)/2$& $\sum{}'2^{\tla+b} M(a,b,c,d)+\sum{}' 2^{a} M(a,b,1,d)$  &  Thm.~\ref{thm:M2121_2101}\\ \hline
$\scriptstyle\gF(1,1;1,1)-2\gF(0,1,1;1)$& $ \sum{}' (2^{a-1}-1)2^b(2^c+1) M(a,b,c,d)$  &  Thm.~\ref{thm:M4W2221_2211}\\ \hline
        & $\sum{}' ( 2^{\tla+b}(2^c+1)+2^a-2^{b+c}-2^c) \z(a,b,c,d)$ & Cor.~\ref{cor:z4W222}, \cite{GuoXi2009}, \cite[(1.10)]{Machide2015b}\\ \hline
${ \frac12\gF(1;1;1;1)-2\gF(0,1;1;1)\atop +2\gF(1,0,1;1)+\gF(0,1;0,1)}$ & $ \sum{}'2^{a+c}(2^{\tla}3^b-3^b-1) M(a,b,c,d)$  &  Thm.~\ref{thm:M4W432_2321}, \cite[(1.12)]{Machide2015b}\\ \hline
\end{tabular}
\end{center}
}
\caption{Weighted/restricted sum formulas in depth at least four.}
\label{Table: Weighted/restricted sum formulas in depth at least four.}
\end{table}

\end{document}